\documentclass[10pt]{amsart}
\usepackage{euscript,graphicx,epstopdf,amscd,amsgen,amsfonts,amssymb,latexsym,
amsmath,amsthm,graphicx,mathrsfs,times,color,overpic,tikz-cd}
\usepackage[normalem]{ulem}
\usepackage{mathtools}
\usepackage{enumerate}
\usepackage{enumitem}
\usepackage{comment}

\numberwithin{equation}{section}

\usepackage[most]{tcolorbox}
\newtcolorbox{problema}[2][]{
                lower separated=false,
                colback=white!80!magenta,
colframe=white, fonttitle=\bfseries,
colbacktitle=white!50!gray,
coltitle=black,
enhanced,
attach boxed title to top left={xshift=0.1cm,yshift=-2mm},
title=#2,#1}

\newtheorem{mtheo}{Theorem}

\newtheorem{theorem}{Theorem}[section]
\newtheorem{lemma}[theorem]{Lemma}
\newtheorem{claim}[theorem]{Claim}
\newtheorem*{claim*}{Claim}

\newtheorem{corollary}[theorem]{Corollary}
\newtheorem{proposition}[theorem]{Proposition}

\theoremstyle{definition}

\newtheorem{remark}[theorem]{Remark}

\newtheorem*{remark*}{Remark}

\newcommand{\eqdef}{\stackrel{\scriptscriptstyle\rm def}{=}}
\newcommand{\eps}{\varepsilon}
\newcommand{\barf}[1]{{\bar{f}^{#1}_0}}

\def\sG{\mathscr{G}}

\def\c{{\rm c}}
\def\s{{\rm s}}

\def\u{{\rm u}}
\def\cs{{\rm cs}}
\def\cu{{\rm cu}}

\def\bN{\mathbb{N}}

\def\bZ{\mathbb{Z}}
\def\bR{\mathbb{R}}

\def\cCC{\EuScript{C}}
\def\cN{\EuScript{N}}

\def\cR{\EuScript{R}}

\def\cS{\EuScript{S}}

\def\cW{\mathscr{W}}

\def\cM{\EuScript{M}}

\def\c{{\rm c}}
\def\s{{\rm s}}
\def\u{{\rm u}}
\def\cs{{\rm cs}}
\def\cu{{\rm cu}}
\def\cF{\mathscr{F}}

\def\uv{\underline v}
\def\uw{\underline w}

\def\fm{\mathfrak m}

\DeclareMathSymbol{\varnothing}{\mathord}{AMSb}{"3F}
\renewcommand{\emptyset}{\varnothing}
\title[Equilibrium states by synchronization]{Equilibrium states by synchronization, symbolic extensions, and factors}
\author[K.~Gelfert]{Katrin Gelfert}
\address{Instituto de Matem\'atica Universidade Federal do Rio de Janeiro, Av. Athos da Silveira Ramos 149, Cidade Universit\'aria - Ilha do Fund\~ao, Rio de Janeiro 21945-909,  Brazil}
\email{gelfert@im.ufrj.br}
\author[D.~Kwietniak]{Dominik Kwietniak}
\address{Faculty of Mathematics and Computer Science, Jagiellonian University in Krak\'ow, ul. \L o\-jasiewicza 6, 30-348 Krak\'ow, Poland}\email{dominik.kwietniak@uj.edu.pl}
\author[Y.~Lima]{Yuri Lima}
\address{Instituto de Matemática e Estatística, Universidade de São Paulo, Rua do Matão, 1010, Cidade Universitária, 05508-090. São Paulo -- SP, Brazil}
\email{yurilima@gmail.com}
\begin{thanks}
{This research has been supported [in part] by CAPES – Finance Code 001, by CNPq-grants 302880/2015-1, 302194/2018-5, and 305327/2022-4, and E-16/2014 INCT/FAPERJ and E-26/211.313/2021 FAPERJ (Brazil).
YL was supported by CNPq/MCTI/FNDCT project 406750/2021-1; FUNCAP grant UNI-0210-00288.01.00/23;
and Instituto Serrapilheira, grant 
``Jangada Din\^{a}mica: Impulsionando Sistemas Din\^{a}micos na Regi\~{a}o Nordeste''. The paper was finalized during the visit of KG to the Jagiellonian University funded by the program Excellence Initiative–--Research University at the
Jagiellonian University in Kraków. DK was also supported by this program. KG is grateful to P.~Carraso, Y.~Coudene, G.~Iommi, S.~Martinchich, R.~Ricks, R.~Ruggiero, B.~Schapira, M.~Todd, and C.~Vasquez  for discussions.}
\end{thanks}
\begin{document}
\begin{abstract}
We combine the two classical topological concepts,  time-preserving topological factors and synchronizing time-changes of a continuous flow, and explore some of their thermodynamic consequences. Particular focus is put on equilibrium states and, in particular, measures of maximal entropy, with emphasis on geodesic flows on rank-one surfaces of nonpositive curvature and their time-preserving expansive topological factors for which we investigate the scaled geometric potentials. 
\end{abstract}

\maketitle

\section{Introduction}

We investigate the geodesic flow $G=(g^t)_{t\in\mathbb{R}}$ of a compact rank-one surface $M$ with nonpositive curvature, concentrating on ergodic measures of $G$ by means of symbolic extensions, topological factors, and time-changes. Understanding thermodynamic properties of $G$ and, in particular, studying its measures of maximal entropy and equilibrium states is a classical yet challenging topic of research. Due to their dynamic relevance, the equilibrium states of the \emph{scaled geometric potential} $q\varphi^{(\u)}$, $q\in\bR$, are of special interest, where 
\begin{equation}\label{phidef}
	\varphi^{(\u)}\colon T^1M\to\bR,\quad
	\varphi^{(\u)}(v)
	\eqdef -\frac{d}{dt}\Big|_{t=0}\log\,\lVert dg^t|_{F^\u_v}\rVert
	= -\lim_{t\to0}\frac1t\log\,\lVert dg^t|_{F^\u_v}\rVert,
\end{equation}
and $F^\u$ denotes a continuous one-dimensional invariant subbundle of $TT^1M$ on which the derivative of the geodesic flow is ``noncontracting'' (see Section \ref{sec:geodesic} for details).

Our main goal is to provide alternative approaches to the following result in \cite{BurCliFisTho:18}. 

\begin{mtheo}\label{theoremgeodesic}
	Let $M$ be a connected compact rank-one Riemannian surface of nonpositive sectional curvature, and let $G=(g^t)_{t\in\mathbb{R}}$ be its geodesic flow on the unit tangent bundle. For every $q<1$, the potential $q\varphi^{(\u)}$ has a unique equilibrium state (with respect to $G$). Moreover, this measure is supported on the set of regular vectors. 
\end{mtheo}

Since \eqref{phidef} is continuous and $G$ is a $C^\infty$ flow, \cite{Newhouse-Entropy} guarantees the existence of equilibrium states for $q\varphi^{(\u)}$ for any $q\in\bR$. The study of measures of maximal entropy (and, more generally, equilibrium states) for geodesic flows has recently gained much attention.  When $G$ is an Anosov flow%
\footnote{For a compact rank-one nonpositively curved surface $M$ its geodesic flow is Anosov if, and only if, every geodesic curve contains at least one point with negative curvature \cite{Ebe:73I-II}.},
  uniqueness follows from \cite{BowRue:75}. The non-Anosov case provides fundamental examples of nonuniformly hyperbolic dynamical systems.  Among the first results for non-Anosov flows, we mention the uniqueness of the measure of maximal entropy in \cite{Kni:98}.   
  
  The original proof of Theorem \ref{theoremgeodesic} in \cite{BurCliFisTho:18} uses a non-uniform version of Bowen's criterion \cite{Bow:75} developed in the discrete case in \cite{CliTho:12} (see \cite{CliTho:21} for later developments).   Our approach to the proof of Theorem \ref{theoremgeodesic} uses {\em Parry’s synchronization technique} \cite{Par:86}. This technique was originally introduced to study hyperbolic attractors and, to the best of our knowledge, it has so far been explored only in uniformly hyperbolic settings (see, for example, \cite{Sam:14}). In essence, synchronization is a specific time-change (or reparametrization) that induces a homeomorphism between the spaces of invariant measures of the flow and its time-changed counterpart and sends a designated equilibrium state to the measure of maximal entropy 
of the time-changed flow. 

Additionally to the context described above, we also develop techniques for studying equilibrium states by synchronization for a special class of expansive flows that emerged in our investigations. We provide more details of this context in Section \ref{sect-expansive-factor}.

The first main result of the article, Theorem \ref{thepro:2}, describes synchronization for general continuous flows and a particular class of potentials. To state it, we need some notation. Given a flow $\Phi=(\phi^s)_{s\in\bR}$
and a function $f$, let $P_{\rm top}(\Phi,f)$ denote the \emph{topological pressure of $f$ with respect to $\Phi$}, see the definition in Section \ref{sec:entpress}. To each positive continuous function $r$, we define
in Section \ref{sec:timcha} the time-changed flow $\Phi_r$. In particular, if $\mu$ is a $\Phi$-invariant measure, then we can define a $\Phi_r$-invariant measure $\mu_r$ by the equation \eqref{eq:measuretime}. If $r$ is a {\em synchronization} for $f$, defined by equation (\ref{eq:T-for-Pnew}) below, then we relate equilibrium state for $f$ with respect to $\Phi$ with measures of maximal entropy for $\Phi_r$.

\begin{mtheo}[Synchronization of equilibrium states]\label{thepro:2}
	Let $\Phi\colon Y\times\bR\to Y$ be a continuous fixed-point free flow on a compact metric space $Y$.
	Suppose $f\colon Y\to\bR$ is a continuous function for which there exists $t>0$ such that 
\begin{equation}\label{eq:T-for-Pnew}
  r\eqdef P_{\rm top}(\Phi,f)-\max_{y\in Y}\frac1t \int_0^tf(\phi^s(y))\,ds>0.
\end{equation}
Let $\Phi_r$ be the time-changed flow of $\Phi$ with respect to $r$. Then an ergodic measure $\mu$ is an equilibrium state for $f$ with respect to $\Phi$ if, and only if, the measure $\mu_r$ defined by equality
\begin{equation}\label{eq:measuretime}
\frac{d\mu_r}{d\mu} \eqdef \frac{r}{\int r\,d\mu}
\end{equation}
is an ergodic measure of maximal entropy for $\Phi_{r}$. In particular, $f$ has a unique equilibrium state with respect to $\Phi$ if, and only if, the flow $\Phi_{r}$ has a unique measure of maximal entropy.
\end{mtheo}

We apply Theorem \ref{thepro:2} for nonuniformly hyperbolic geodesic flows $\Phi=G$ as discussed in the beginning
of this introduction, observing that condition \eqref{eq:T-for-Pnew} is satisfied for the potential $f=q\varphi^{(\u)}$
if, and only if, $q<1$. In other terms, $f=q\varphi^{(\u)}$ is a \emph{hyperbolic potential}%
\footnote{Without further structure, the term \emph{hyperbolic} in the general setting of Theorem \ref{thepro:2}  \emph{a priori} does not make sense. It is motivated by \cite{InoRiv:12}. It will be justified in Lemma \ref{lem:equi}.}; 
see Section \ref{ssechyppot} for a detailed discussion. 

Once Theorem \ref{thepro:2} is applied, we are then interested in analyzing when the time-changed flow has a unique measure of maximal entropy. As a matter of fact, we do not apply Theorem \ref{thepro:2} directly to $\Phi$,
but rather to either a (symbolic) extension or a (topological) factor of $\Phi$. More precisely, we pursue the following two lines of investigation:
\begin{enumerate}[leftmargin=0.8cm, label=(A\arabic*)]
\item\label{aa1} We consider a symbolic extension $\Theta$ of the geodesic flow $G$. The function
$f$ lifts to a function $\widehat f$ on the symbolic space. For some appropriately
chosen function $\widehat r$, we consider the time-change $\Theta_{\widehat r}$ and then prove that this latter system
has a unique measure of maximal entropy. The conclusion is that $\widehat f$ has a unique equilibrium state,
and consequently (using properties of the symbolic coding) $f$ itself also has a unique equilibrium state.
There is a similarity of this argument with the one developed in \cite{LimPol:}, where the authors provide
a proof of Theorem \ref{theoremgeodesic} using symbolic dynamics. The difference, which we want to emphasize,
is that our proof also uses synchronization, thus reducing the argument to analyzing the uniqueness
of measures of maximal entropy, a problem that is usually easier to solve.
This is done in Section \ref{ssectime-preserving-extension}. 
\item\label{aa2} Starting with a (non-necessarily expansive) geodesic flow $G$, we consider a topological factor $\Psi$ of $G$ that is expansive by ``collapsing'' each flat strip%
\footnote{A \emph{flat strip} is a totally geodesic and isometric immersion of the product of $\bR$ with an interval.}
 to a single orbit. Let $X$ be the resulting three-dimensional manifold and $\pi\colon T^1M\to X$ the associated topological factor map. The existence of this factor is demonstrated in \cite{GelRug:19}. 
Recall that flat strips are the ``only obstructions to expansivity'' for the geodesic flow $G$.
Passing to the factor sends measures with positive entropy (in particular the equilibrium states for $q\varphi^{(\u)}$,
$q<1$) to measures with positive entropy. Then, for some appropriately defined function $\widetilde r$,
we make a time-change $\Psi_{\widetilde r}$ of $\Psi$ and obtain the following commutative diagram:
\begin{equation}\label{eqdiagram}
\begin{tikzcd}
    G \arrow{rr}{\text{time-change by $ r$}} \arrow[swap]{d}{\pi}
    	& & G_{ r} \arrow{d}{\pi}\\
    \Psi \arrow{rr}{\text{time-change by $\widetilde r$}}
    	& & \Psi_{\widetilde r}&
  \end{tikzcd}
  \quad\quad
  \widetilde r\eqdef r\circ\pi^{-1}
\end{equation}
Then $q\varphi^{(\u)}$ has a unique equilibrium state for $G$ if, and only if, 
$\Psi_{\widetilde r}$ has a unique measure of maximal entropy. The discussion of this approach
is made in Section \ref{sect-expansive-factor}.
\end{enumerate}

Regarding approach \ref{aa2}, unfortunately we do not have a direct proof (not relying on approach \ref{aa1})
that $\Psi_{\widetilde r}$ has a unique measure of maximal entropy.
However, the symbolic extension technique applies to $\Psi_{\widetilde r}$ and yields the result below,
also treated in Section \ref{sect-expansive-factor}.

\begin{corollary}\label{corexpansive}
	Let $\Psi$ be the factor flow from \cite{GelRug:19} discussed above and $\pi\colon T^1M\to X$ be the associated topological factor map. Given $q<1$, consider the time-change $\Psi_{\widetilde r}$ of $\Psi$ associated to the function
\[
	\widetilde r\eqdef r\circ\pi^{-1}\colon X\to\bR_{>0},\quad
	r(v)\eqdef P_{\rm top}(G,q\varphi^{(\u)})-\int_0^1q\varphi^{(\u)}(g^s(v))\,ds.
\]	
Then $\Psi_{\widetilde r}$ has a unique measure of maximal entropy.
\end{corollary}

When trying to implement a direct proof that $\Psi_{\widetilde r}$ has a unique measure of maximal entropy, 
a natural approach would be to prove that 
$\Psi_{\widetilde r}$ is expansive, topologically mixing and has local product structure. If this is the case,
then \cite{Fra:77} implies the aforementioned uniqueness
(see also \cite{Oka:97}, and \cite[Corollary 6.6]{GelRug:19}). The main difficulty is that, even knowing
that $\Psi$ is topologically mixing and has local product structure, these properties {\em do not}
behave nicely under time-changes. 

Having in mind the importance of properties that remain true under time-changes, in Section \ref{sec:topdyn}
we consider some of them, including expansivity and the pseudo-orbit-tracing property.
Since $\Psi$ has these properties, it follows that $\Psi_{\widetilde r}$ also has, 
thus proving that it belongs to the class of flows that correspond to TA (\emph{topologically Anosov}) homeomorphisms, as defined in \cite[page 90]{AokHir:94}. 

At this point, checking that $\Psi_{\widetilde r}$ is topologically mixing and has local product structure
could be done via a ``topological dynamics'' proof, that is, not relying on further synchronizations. In this direction, a natural approach would be to study expansive flows in parallel with the study of expansive homeomorphisms. However, such a ``translation'' is not always feasible and, besides technical challenges,
it also presents intrinsic novelties, as we now explain.

Let us first recall known results about an expansive homeomorphism $T$ of a compact metric space $X$. 
Bowen proved that if $T$ has the specification property,
then it has a unique measure of maximal entropy \cite{Bow:75}. In general, specification is obtained
as a consequence of topologically mixing and the pseudo-orbit-tracing property. For expansive homeomorphisms, local product structure and the pseudo-orbit-tracing-property are equivalent (see \cite{Omb:86}). 
By a result nowadays known as \emph{Bowen's spectral decomposition theorem} (see \cite[Theorem 3.1.11(2)]{AokHir:94}), if $T$ is topologically transitive and has the pseudo-orbit-tracing property but fails to be topologically mixing, then we can write $X=X_0\cup \cdots\cup X_{n-1}$, where $X_0,X_1,\ldots X_{n-1}$ are  disjoint nonempty closed sets such that $T(X_i)=X_{i+1}$ and $T^n|_{X_i}\colon X_i\to X_i$ is topologically mixing for every $i=0,\ldots,n-1$ (indices are considered modulo $n$). Combining these results, it follows that every transitive expansive homeomorphism with the pseudo-orbit-tracing property has a unique measure of maximal entropy.

In view of the above discussion, one may suspect of the following dichotomy: if $\Phi$ is a continuous transitive expansive fixed point-free flow on a compact metric space with the pseudo-orbit-tracing property, then 
\begin{enumerate}[leftmargin=0.6cm ]
\item[(i)] either $\Phi$ is topologically mixing (and hence has the specification property, see \cite{GelRug:19});
\item[(ii)] or $\Phi$ is a constant time suspension of a transitive expansive homeomorphism with the pseudo-orbit-tracing property.
\end{enumerate}
This dichotomy indeed holds for fixed-point free Axiom A flows \cite[(3.2) Theorem]{Bow:72}. More generally, we have the following result. As it follows \emph{verbatim} from the proof in \cite{Bow:72}, we state it without proof. 

\begin{mtheo}\label{thepropreBru}
Let $\Phi\colon X\times\bR\to X$ be a continuous expansive fixed-point free flow on a compact metric space $X$ that is topologically transitive and has local product structure. Then exactly one of the following properties hold:
\begin{enumerate}[leftmargin=0.6cm ]
\item[(i)] For every periodic point $p=\phi^\tau(p)\in X$ it holds $\overline{\cW^\u(\Phi,p)}= X$, where
\[
	\cW^\u(\Phi,p)
	\eqdef \{y\in X\colon d(\phi^{-t}(y),\phi^{-t}(p))\to0\text{ as }t\to\infty\}.
\]
\item[(ii)] There exist $t>0$ and $p=\phi^\tau(p)\in X$ such that $\Phi$ is topologically conjugate to the time-$t$ suspension flow over the compact metric space $Y\eqdef \overline{\cW^\u(\Phi,p)}$.
Moreover, for every $x\in X$ there is $s\in(0,t]$ such that $\phi^s(x)\in Y$.
Finally, $t$ is the smallest positive number such that $\phi^t(Y)\cap Y\ne\emptyset$.
\end{enumerate}
\end{mtheo}

We also establish the following result.

\begin{mtheo}[Pseudo-orbit-tracing and orbit closing property]\label{thmprolem:pseudoorbittracing}
	Let $\Phi\colon X\times\bR\to X$ be a continuous fixed-point free flow on a compact metric space that is expansive and has local product structure. Then $\Phi$ has the pseudo-orbit-tracing property and the orbit closing property.
\end{mtheo}

Theorems \ref{thepropreBru} and \ref{thmprolem:pseudoorbittracing} are topological in nature and does not relate to synchronization. Unfortunately, it is still unclear if, for expansive fixed-point free flows, the pseudo-orbit-tracing property implies local product structure. If so, then approach \ref{aa2} would lead to a new proof of Theorem \ref{theoremgeodesic}. It would also demonstrate, for expansive fixed-point free flows, that local product structure is preserved under time-changes (in general, the preservation of local product structure under time-changes of flows seems to be a rather delicate problem, see e.g. \cite{Bru:95}).

The structure of this paper is as follows. In Section \ref{sec:entpress} we recall some key concepts from ergodic theory and thermodynamic formalism. Towards approach \ref{aa1}, in Section \ref{sect-MME-symbolic} we study symbolic extensions, while for approach \ref{aa2} we consider topological factors in Section \ref{secmeamaxent}.  In Section \ref{sec:timcha} we discuss time-changes and synchronization and prove Theorem \ref{thepro:2}. In Section \ref{sec:topdyn} we recall key concepts in topological dynamics and prove Theorem \ref{thmprolem:pseudoorbittracing}. In Section \ref{sec:geodesic} we discuss the particular case of geodesic flows and prove Theorem \ref{theoremgeodesic} implementing approach
\ref{aa1} (see Section \ref{ssectime-preserving-extension}) and approach \ref{aa2} (see Section \ref{sect-expansive-factor}). In this latter section, we also prove Corollary \ref{corexpansive}.

\section{Topological pressure and entropy}\label{sec:entpress}

We briefly recall some concepts from thermodynamic formalism. Let $(Y,d)$ be some compact metric space. 
Given a continuous map $\phi\colon Y\to Y$, we denote by $\cM(\phi)$ the space  of all  $\phi$-invariant Borel probability measures and endow it with the weak$\ast$ topology. By $\cM_{\rm erg}(\phi)$ we denote the set of ergodic measures in $\cM(\phi)$.  

Given a continuous flow $\Phi\colon Y\times\bR\to Y$, we write $\cM(\Phi)\eqdef\bigcap_{t\in\bR}\cM(\phi^t)$ for the set of all $\Phi$-invariant probability measures on $Y$ and $\cM_{\rm erg}(\Phi)\subset\cM(\Phi)$ stands for the subset of all ergodic measures. If $\mu\in\cM_{\rm erg}(\Phi)$, then for all but countably many $t$ it holds that $\mu\in\cM_{\rm erg}(\phi^t)$ (\cite{PugShu:71} or \cite[Theorem 3.3.34]{FisHas:19}).

Given $t>0$ and $\varepsilon>0$, we say a set $E\subset Y$ is \emph{$(t,\varepsilon)$-separated} if $y_1$, $y_2\in E$, $y_1\ne y_2$ implies that $d( \phi^s(y_1), \phi^s(y_2))\ge \varepsilon$ for some $s\in[0,t]$ (see \cite[Chapter 9]{Wal:82} or \cite[Chapter 4.2]{FisHas:19} for details). Given a continuous function $f\colon Y\to\bR$, the  \emph{topological pressure} of $f$ with respect to the flow $\Phi$ is defined by
\[
P_{\rm top}(\Phi,f) \eqdef
\lim_{\varepsilon\to 0^+}\limsup_{t\to\infty}\frac{1}{t}\log
\sup\left\{\sum_{y\in E}e^{\int_0^tf( \phi^s(y))\,ds}\right\},
\]
where the supremum is taken over all $(t,\varepsilon)$-separated sets
$E\subset Y$;  the \emph{topological entropy} of $\Phi$, denoted $h_{\rm top}(\Phi)$ is the pressure of the potential $f\equiv0$,  that is, $h_{\rm top}(\Phi) = P_{\rm top}(\Phi,f\equiv0)$.

\subsection{Equilibrium states and measures of maximal entropy}

Given a map $\phi\colon Y\to Y$ and some $\phi$-invariant probability measure $\mu$, denote by $h_\mu(\phi)$ the metric entropy of $\mu$ with respect to $\phi$.  
The following variational principle holds (see \cite[Theorem 9.10]{Wal:82}, or \cite[Theorem 4.3.8]{FisHas:19})
\begin{equation}\label{varprinc}
	P_{\rm top}(\Phi,f)
	=	\sup_{\nu\in \cM(\Phi)}
	\left(h_\mu( \phi^1)+\int f \,d\mu\right)
	=	\sup_{\mu\in \cM_{\rm erg}(\Phi)}
	\left(h_\mu( \phi^1)+\int f \,d\mu\right).
\end{equation}
A measure that realizes the supremum in (\ref{varprinc}) is an \emph{equilibrium state} for the potential $f$ with respect to $\Phi$. An equilibrium state for the constant potential $f\equiv 0$ is a \emph{measure of maximal entropy}. 

Pressure and equilibrium states with respect to a continuous map on $Y$ are analogously defined, see \cite{Wal:82}.

By Abramov's formula (see ~\cite{Abr:59} or \cite[Corollary 4.1.10]{FisHas:19}), for every  $\mu\in\cM(\Phi)$ and $t\in\bR$ we have
\begin{equation}\label{abra}
  h_\mu( \phi^t) = | t|\, h_\mu( \phi^1).
\end{equation}
One calls $h_\mu(\Phi)\eqdef h_\mu( \phi^1)$ the \emph{metric entropy} of $\mu$ with respect to the \emph{flow} $\Phi$.

Let us denote
\begin{equation}\label{eq:T-for-Pintegral}
		\bar f_0^T (y)
	\eqdef \int_0^Tf(\phi^s(y))\,ds.
\end{equation}
We call a continuous function $f\colon Y\to\bR$  \emph{hyperbolic with respect to $\Phi$} or simply \emph{hyperbolic} if there exists $T>0$ such that
\begin{equation}\label{eq:T-for-Pdefhyperbolic}
  P_{\rm top}(\Phi,f)-\frac1T\max_{y\in Y} \barf{T}(y)>0.
\end{equation}
In Section \ref{ssechyppot}, we will further explore hyperbolic potentials in the context of geodesic flows studied in this paper. 

Since $\mu\in\cM(\Phi)$ is $\phi^t$-invariant for every $t>0$, we have $\int f\circ\phi^t\,d\mu=\int f\,d\mu$.  Hence for any $T>0$ it holds
\begin{equation}\label{eq:int-of-barphi}
  \int \barf{T}\,d\mu
	= \int \int_0^Tf\circ \phi^t\,dt\,d\mu
	= \int_0^T\int f\circ \phi^t\,d\mu\,dt
	= T\int f\,d\mu.
\end{equation}
It follows that for every $T>0$ we have
\begin{equation}\label{eq:P}
	P_{\rm top}(\Phi,f)
	= P_{\rm top}( \Phi,\frac1T\bar f_0^T)
	= P_{\rm top}( \phi^1,\frac1T\bar f_0^T),
\end{equation}
where the last term  is the classically defined topological pressure of the potential $\frac1T\bar f_0^T$ with respect to the homeomorphism $ \phi^1$ (see \cite[Definition 9.3]{Wal:82}).

We also note the following simple lemma for later reference.

\begin{lemma}\label{lemequalii}
Let $\Phi$ be a flow on a compact metric space $Y$. Assume $f\colon Y\to\bR$ is a continuous potential, $T>0$, and $\barf{T}$ is given by \eqref{eq:T-for-Pintegral}.
Then
\begin{equation}\label{eq:InoRiv}
	\lim_{T\to\infty}\max_{y\in Y}\frac1T \barf{T}(y)
	= \inf_{T>0}\max_{y\in Y}\frac1T \barf{T}(y)
	= \max_{\mu\in\cM(\Phi)} \int f\,d\mu
	\le P_{\rm top}(\Phi, f).
\end{equation}
\end{lemma}

\begin{proof}
The inequality in \eqref{eq:InoRiv} is a consequence of the variational principle \eqref{varprinc}.
Continuity of the functions $\mu\mapsto\int f\,d\mu$ and $y\mapsto \barf{T}(y)$  together with compactness of their domains guarantee that the maxima in \eqref{eq:InoRiv} are indeed attained.

Given $\mu\in\cM(\Phi)$, it follows from \eqref{eq:int-of-barphi} that for every $T>0$ we have
\[
	\int f\,d\mu
	= \frac 1T\int \barf{T}\,d\mu
	\le \max_{y\in Y} \frac1T\barf{T}(y).
\]
Taking first the supremum over $\mu\in\cM(\Phi)$ and then the infimum over $T>0$, gives
\begin{equation}\label{eq:finalineq}
	\max_{\mu\in\cM(\Phi)}\int f\,d\mu
	\le \inf_{T>0}\max_{y\in Y} \frac1T\barf{T}(y).
\end{equation}
For each $T>0$ we let $y_T$ to be a point in $Y$ so that
\[
	\frac1T\barf{T}(y_T)
	= \max_{y\in Y}\frac1T\barf{T}(y)
\]
and define
\[
	\mu_T
	\eqdef \frac1T\int_0^T\delta_{\phi^t(y_T)}\,dt.
\]
Observe that $\mu_T$ is a Borel probability measure on $Y$. Note that
\begin{equation}\label{eq:int-ev}
	\int f\,d\mu_T
	=\frac1T\barf{T}(y_T).
\end{equation}
Choose a sequence $(T_k)_{k=1}^\infty$ with $ T_k\to\infty$ as $k\to\infty$, such that
\begin{equation}\label{eq:lim}
\lim_{k\to\infty}\frac{1}{T_k}\barf{T_k}(y_{T_k})
	= \limsup_{T\to\infty} \frac{1}{T}\barf{T}(y_{T})
	= \limsup_{T\to\infty}\left(\max_{y\in Y}\frac1T\barf{T}(y)\right).
\end{equation}	
Passing, if necessary to a subsequence, we assume that $(\mu_{T_k})_{k=1}^\infty$ converges in the weak$^\ast$ topology to a measure $\nu$ which is easily seen to be $\Phi$-invariant. Hence,  the continuity of $ f$, the fact that $(\mu_{T_k})_{k=1}^\infty$ converges to $\nu$ in the weak$^\ast$ topology,  \eqref{eq:lim} and \eqref{eq:int-ev} imply
\[
	\int f\,d\nu
	= \lim_{k\to\infty}\int f\,d\mu_{T_k}
	= \lim_{k\to\infty}\frac{1}{T_k}\bar f_0^{T_k}(y_{T_k})
	= \limsup_{T\to\infty}\sup_Y\frac{1}{T}\bar f_0^T.
\]
Together with \eqref{eq:finalineq}, this proves all the equalities in \eqref{eq:InoRiv}.
\end{proof}

\subsection{Measures of maximal entropy for (symbolic) extensions}\label{sect-MME-symbolic}

We wish to use symbolic extensions to conclude the uniqueness of the measure of maximal entropy of some
flows. In general, a symbolic extension is neither injective nor surjective, but if it is finite-to-one on subsets
that support the relevant invariant measures, then we can conclude the referred statement.
We start this section with a general result in this context, inspired by \cite[Proposition 13.2]{Sarig-JAMS}.

Let $\Psi\colon X\times \bR\to X$ and $\Phi\colon Y\times\bR\to Y$ be continuous fixed-point free flows,
and let $\mu_1,\mu_2$ be $\Psi$-invariant probability measures. 
We say that a measurable map $\pi\colon Y\to X$ is a \emph{finite-to-one extension for $\mu_1$ and $\mu_2$} 
or that $\Psi$ is a \emph{finite-to-one factor} of $\Phi$ for $\mu_1$ and $\mu_2$ by $\pi$ if the following holds: 
we have $\pi\circ\Phi=\Psi\circ\pi$ and there are measurable sets $Y^\#\subset Y$ and $X^\#\subset X$ such that:
\begin{enumerate}[label=$\circ$]
\item If $\nu$ is a $\Phi$-invariant probability measure, then $\nu(Y^\#)=1$.
\item $\mu_1(X^\#)=\mu_2(X^\#)=1$.
\item $X^\#=\pi(Y^\#)$ and the restriction $\pi|_{Y^\#}\colon Y^\#\to X^\#$ is finite-to-one.
\end{enumerate}

\begin{proposition}\label{prop-finite-to-one-extension}
Let $\Psi\colon X\times \bR\to X$ and $\Phi\colon Y\times\bR\to Y$ be continuous fixed-point free flows.
Let $\mu_1,\mu_2$ be $\Psi$-invariant probability measures that are equilibrium states of a continuous potential $f\colon X\to\bR$, and assume that $\pi\colon Y\to X$ is a finite-to-one extension for $\mu_1$ and $\mu_2$. Then for $i=1,2$ there exists a $\Phi$-invariant probability measure $\nu_i$ which is an equilibrium state for the function $f\circ \pi$ and satisfies $\pi_\ast\nu_i=\mu_i$. 
\end{proposition}

\begin{proof}
The following procedure for lifting $\mu$ to $Y$ is based on the technique described in \cite[Proof of Proposition 13.2]{Sarig-JAMS}. 
Fix $i\in\{1,2\}$ and write $\mu=\mu_i$. Since $\mu(X^\#)=1$, the measure
\[
	\nu
	\eqdef \int_{X^\#} \left(\frac{1}{|\pi^{-1}(x)\cap Y^\#|}\sum_{y\in \pi^{-1}(x)\cap Y^\#}\delta_y\right) d\mu(x)
\]
is a well-defined $\Phi$-invariant probability measure. Since the restriction $\pi|_{Y^\#}\colon Y^\#\to X^\#$ is finite-to-one, the Abramov-Rokhlin formula implies that
$h_\nu(\Phi)=h_\mu(\Psi)$.

It remains to prove that $\nu$ is an equilibrium state for $f\circ\pi$. For that, we need to project measures. Let $\xi$ be a $\Phi$-invariant probability measure. By assumption, $\xi(Y^\#)=1$. Then $\eta\eqdef\pi_\ast\xi$ is a $\Psi$-invariant probability measure. It follows from our assumption that $\eta(X^\#)=\xi(\pi^{-1}(X^\#))=\xi(Y^\#)=1$. By the Abramov-Rokhlin formula, it follows that $h_{\eta}(\Psi)=h_\xi(\Phi)$. Hence, 
\[
	h_\xi(\Phi)+\int_Y (f\circ\pi)\,d\xi
	=h_\eta(\Psi)+\int_X f\,d\eta \leq P_{\rm top}(\Psi,f).
\]
As $\xi$ was arbitrary, this implies that $P_{\rm top}(\Phi,f\circ \pi)\leq P_{\rm top}(\Psi,f)$. But since
\[
	h_{\nu}(\Phi)+\int_Y(f\circ\pi)\,d\nu
	=h_{\mu}(\Psi)+\int_X f\,d\mu =P_{\rm top}(\Psi,f),
\]
it follows that $P_{\rm top}(\Phi,f\circ \pi)= P_{\rm top}(\Psi,f)$. Hence, $\nu$ is an equilibrium state for $f\circ\pi$.
\end{proof}

Now we introduce topological Markov flows, which are the symbolic extensions we are interested in.
Let $\sG=(V,E)$ be an oriented graph whose vertex set is $V$
and edge set is $E$. An edge from $v\in V$ to $w\in V$ denote edges by $v\to w$ and assume that $V$ is countable. A \emph{topological Markov shift} (\emph{TMS}) is a pair $(\Sigma,\sigma)$, where 
\[
	\Sigma\eqdef
	\{\bZ\text{-indexed paths on }\sG\}
	= \{\uv=(v_n)_{n\in\bZ}\in V^\bZ\colon v_n\to v_{n+1},\,\forall n\in\bZ\}
\]
is the \emph{symbolic space} and $\sigma\colon\Sigma\to\Sigma$, $\sigma(\uv)_n=v_{n+1}$, is the \emph{left shift}. 
We endow $\Sigma$ with the distance $d(\uv,\uw)\eqdef \exp(-\inf\{|n|\colon v_n\ne w_n\})$.
We only consider TMS that are \emph{locally compact}, that is, for all $v\in V$ the number of ingoing edges $u\to v$ and outgoing edges $v\to w$ is finite.

Given some \emph{roof function} $R\colon\Sigma\to\bR_{>0}$, consider the \emph{suspension space} 
\[
	\Sigma_R
	\eqdef (\Sigma\times\bR_{\ge0})/_\sim,
\]
where $\sim$ is the equivalence relation identifying $(\uv,s)$ with $(\sigma(\uv),s-R(\uv))$. Unless stated otherwise, we agree to represent each class by its \emph{canonical representation} $(\uv,s)$, where $s\in[0,R(\uv))$. We equip this space with the Bowen-Walters metric $d_{\Sigma,R}$ (see \cite[Section 4]{BowWal:72} for details). The \emph{suspension} of $(\Sigma,\sigma)$ by $R$ is the flow which for $t$ small such that $0\le s+t<R(\uv)$ is defined by 
\[
	\Theta=\Theta_{\Sigma,R}:\Sigma_R\times\bR\to\Sigma_R,\quad
	\theta^t(\uv,s)
	\eqdef (\uv,s+t)
\]
and extended to all $t\in\bR$ accordingly. This flow is also called the \emph{topological Markov flow} (\emph{TMF}) defined by a TMS $(\Sigma,\sigma)$ and $R$. 

Let $\fm$ be the Lebesgue measure on $\bR$. If $\nu$ is a $\sigma$-invariant ergodic probability measure, then 
\begin{equation}\label{def-lift-measure}
	\mu=\mu(\nu)
	\eqdef \frac{(\nu\times\fm)|_{\Sigma_R}}{(\nu\times\fm)(\Sigma_R)}
\end{equation}
is a $\Theta$-invariant ergodic probability measure.

\begin{lemma}\label{lemcompl}
	Let $R\colon\Sigma\to\bR_{>0}$ be a continuous function which is uniformly bounded away from zero and infinity. Then the map $\nu\mapsto\mu(\nu)$ in \eqref{def-lift-measure} is a bijection between $\cM(\sigma)$ and $\cM(\Phi)$.
\end{lemma}	 

\begin{proof}
	Let $\mu$ be a $\Phi$-invariant Borel probability measure. Let $\widetilde\nu\eqdef\pi_\ast\mu$ denote its projection to $\Sigma$ by the canonical projection $\varrho\colon\Sigma_R\to\Sigma$ defined by $\varrho(\uv,s)\eqdef\uv$. The collection of sets $\cCC\eqdef\{\{\uv\}\times [0,R(\uv)]\colon\uv\in\Sigma\}$ provides a measurable partition of $\Sigma_R$. As $\mu$ is $\Theta$-invariant, its conditional measures on the elements of $\cCC$ are the normalized Lebesgue measures on $[0,R(\uv)]$, for $\widetilde\nu$-almost every $\uv$. If follows from our hypothesis that $1/R$ is bounded and hence $\widetilde\nu$-integrable. Consider the measure $\nu$ on $\Sigma$ defined by
\[
	\nu(A)
	\eqdef \frac1N\int_A\frac{1}{R(\uv)}\,d\widetilde\nu(\uv),
	\quad\text{ where }\quad
	N
	\eqdef \int \frac1R\,d\widetilde\nu.
\]	
It is easy to see that $\nu$ is a Borel probability measure. By definition, $\int R\,d\nu=(\nu\times\fm)(\Sigma_R)$. 
Moreover, $\nu$ is $\sigma$-invariant and
\[
	\mu
	= \frac{1}{\int Rd\nu}(\nu\times\fm)|_{\Sigma_R}
	= \frac{(\nu\times\fm)|_{\Sigma_R}}{(\nu\times\fm)|(\Sigma_R)}.
\] 
This completes the proof.
\end{proof}

Given a measurable function $f\colon\Sigma_R\to\bR$, let 
\[
	\Delta f\colon\Sigma\to\bR,\quad
	\Delta f(\uv)
	\eqdef \int_0^{R(\uv)}f(\uv,s)\,ds.
\]
By the Abramov formulas,
\begin{equation}\label{eqAbramov}
	h_{\mu}(\Theta)
	= \frac{h_\nu(\sigma)}{\int R\,d\nu}
		\quad\text{ and }\quad
	\int f d\mu
	= \frac{\int\Delta f\,d\nu}{\int R\,d\nu}.
\end{equation}

\begin{theorem}\label{thelemkey}
Let $\Theta$ be a TMF defined by a TMS $(\Sigma,\sigma)$ and a continuous roof function $R\colon\Sigma\to\bR_{>0}$ which is uniformly bounded away from zero and infinity. Let $f\colon\Sigma_R\to\bR$ be a continuous function. Then:
\begin{enumerate}
\item[(i)] $P_{\rm top}(\Theta,f)=c$ if, and only if, $P_{\rm top}(\sigma,\Delta f-cR)=0$. 
\item[(ii)] For every $\nu\in \cM(\sigma)$ the measure $\mu(\nu)$ defined by \eqref{def-lift-measure} is an equilibrium state for $f$ with respect to $\Theta$ if, and only if, $\nu$ is an equilibrium state for $\Delta f-P_{\rm top}(\Theta,f)R$ with respect to $\sigma$.
\end{enumerate}
In particular, $\mu(\nu)$ is a measure of maximal entropy for $\Theta$ if, and only if,  $\nu$ is an equilibrium state for $-h_{\rm top}(\Theta)R$ with respect to $\sigma$.
\end{theorem}

\begin{proof}
Using Lemma \ref{lemcompl} we easily see that \eqref{eqAbramov} yields
\[
	\sup_\nu\left(h_\nu(\sigma)+\int(\Delta f-c R)\,d\nu\right)
	= \sup_\nu\left(\frac{h_\nu(\sigma)}{\int R\,d\nu}+\frac{\int\Delta f\,d\nu}{\int R\,d\nu}-c\right)
		\int R\,d\nu,
\]
where both suprema are taken over all $\sigma$-invariant ergodic probability measures $\nu$. This immediately implies both assertions.
\end{proof}

We now state a criterion for the uniqueness of the measure of maximal entropy for 
TMF. Recall that a TMS $(\Sigma,\sigma)$ with associated vertex set $V$ is  \emph{irreducible} 
or \emph{topologically transitive} if
\[
	\forall u,w\in V,\,\exists \uv=(v_k)_{k\in\mathbb{Z}}\in \Sigma\text{ and } \exists n\in\bN\text{ such that } v_0=u \text{ and } v_n=w.
\]	
Let $\Sigma_R$ be a TMF associated to $(\Sigma,\sigma)$ and $R$. Accordingly, we call $\Sigma_R$
\emph{irreducible} if $(\Sigma,\sigma)$ is irreducible as above.

\begin{corollary}\label{corTMFmme}
Let $(\Sigma,\sigma)$ be an irreducible TMS and $R\colon\Sigma\to\bR_{>0}$ a H\"older continuous finite roof function. Then the associated TMF $\Sigma_R$ has at most one measure of maximal entropy.
\end{corollary}

\begin{proof}
Assume that the TMF has a measure $\mu$ of maximal entropy. By Lemma \ref{lemcompl}, there is some $\sigma$-invariant probability measure $\nu$ such that $\mu=\mu(\nu)$ as in \eqref{def-lift-measure}. By Theorem \ref{thelemkey}, $\nu$ is an equilibrium state for $-h_{\rm top}(\Theta_{\Sigma,R})R$. Since $R$ is H\"older continuous, by \cite[Theorem 1.1]{BuzSar:03} this equilibrium state is unique.
\end{proof}

\begin{corollary}
Let $\Psi\colon X\times\bR\to X$ be a continuous fixed-point free flow. 
If $\Psi$ has two measures of maximal entropy $\mu_1$ and $\mu_2$,
then $\Psi$ is not a finite-to-one factor of an irreducible TMF for $\mu_1$ and $\mu_2$.
\end{corollary}

\begin{proof}
By contradiction, assume that $\pi\colon\Sigma_R\to X$ is a finite-to-one extension 
for $\mu_1$ and $\mu_2$, where $\Sigma_R$ is irreducible. Let $\Theta$ be the suspension
flow on $\Sigma_R$.
By Proposition \ref{prop-finite-to-one-extension}, $\mu_1$ and $\mu_2$ lift to distinct 
measures of maximal entropy for  $\Theta$, which contradicts
Corollary \ref{corTMFmme}. This completes the proof.
\end{proof}

\subsection{Measures of maximal entropy for factors}\label{secmeamaxent}

A continuous flow $\Psi\colon X\times\bR \to X$ is a \emph{time-preserving topological factor} of a continuous flow $\Phi\colon Y\times\bR\to Y$  if there exists a continuous surjective map $\pi\colon Y\to X$ such that for every $t\in\bR$ we have
\[
	\pi\circ \phi^t
	= \psi^t\circ \pi.
\]
We say that $\Psi$ is an \emph{entropy-preserving topological factor} of $\Phi$ on $\cN\subset\cM_{\rm erg}(\Phi)$ if%
\footnote{Here, $h_{\rm top}(\phi^1,A)$ refers to the topological entropy of a set $A\subset Y$ with respect to $\phi^1$ (we refer to \cite{Bow:73b} for the definition).}
  for every $\mu\in\cN$ we have:	
\begin{enumerate}[label=\textnormal{(\arabic*)}]
	\item $h_{\rm top}(\phi^1,\pi^{-1}(\pi(y)))=0$ for $\mu$-almost every $y\in Y$;\label{c:A} 
	\item $\pi_\ast\mu(\{\pi(y)\colon \pi^{-1}(\pi(y))=\{y\}\})=1$. \label{c:B}
\end{enumerate}

Below we formulate a version of~\cite[Theorem 1.5]{BuzFisSamVas:12} for flows (the original in \cite{BuzFisSamVas:12} is stated for discrete systems). For completeness, we provide a full proof, which is based on the essential observation that it suffices to study the time-$1$ map of the flow and hence to reduce the arguments to the ones in \cite{BuzFisSamVas:12}. Note that \cite{BuzFisSamVas:12} also discusses the explicit \emph{construction} of the measure of maximal entropy as limit measure of the usual convex combination of periodic orbit-measures. This construction can be done analogously but will not be discussed here (see also \cite[Theorem 8.3.6]{FisHas:19}).

\begin{proposition}\label{prothe:2}
	Let $\Psi\colon X\times\bR\to X$ and $\Phi\colon Y\times\bR\to Y$ be two continuous fixed-point free flows on compact metric spaces $X$ and $Y$, respectively. Assume that $\Psi$ is a time- and entropy-preserving topological factor of $\Phi$ on $\cN\subset\cM_{\rm erg}(\Phi)$ through a continuous surjective map $\pi\colon Y\to X$. 
If $\nu$ is the unique measure of maximal entropy in $\pi_\ast\cN$ (with respect to $\Psi$), that is, $\nu\in\cM_{\rm erg}(\Psi)$ is the unique measure satisfying $\nu\in\pi_\ast\cN$ and
\[
	h_\nu(\psi^1)
	= h_{\rm top}(\Psi),
\]	 
then $\nu$ lifts uniquely to a  measure of maximal entropy in $\cN$ (with respect to $\Phi$).
\end{proposition}

\begin{proof}
	Let $\nu$ be as above.  Note that $\cN\subset\cM_{\rm erg}(\Phi)$ implies that $\nu$ is $\Psi$-ergodic. Consider a $\nu$-generic point $x\in X$. Let $y\in \pi^{-1}(x)$. Consider the sequence
\[
	\mu_{y,T}
	\eqdef \frac1T\int_0^T\delta_{\phi^t(y)}\,dt, \quad T=1,2,3,\ldots.
\]	
This sequence has a subsequence which in the weak$\ast$ topology converges to a $\Phi$-invariant probability measure as $T\to\infty$. Let $\mu$ be any of its ergodic components. It is easy to see that $\pi_\ast\mu=\nu$. As in \cite{BuzFisSamVas:12}, we now prove that $\mu$ is the unique measure satisfying $\pi_\ast\mu=\nu$ and that it is of maximal entropy.

\begin{claim}\label{cla:parrotc}
	The measure $\mu$ is of maximal entropy (with respect to $\Phi$) and $h_\mu(\Phi)=h_\nu(\Psi)$.
\end{claim}	
	
\begin{proof}
	Since the flow $\Phi$ is a time-preserving extension of $\Psi$ and $\pi_\ast\mu=\nu$ we get
\[
	h_\mu(\phi^1)
	\ge h_\nu(\psi^1).
\]	
On the other hand, by \cite{LedWal:77} it holds
\[
	\sup_{\mu'\colon\pi_\ast\mu'=\nu}h_{\mu'}(\phi^1)
	= h_\nu(\psi^1)+\int_Xh_{\rm top}(\phi^1,\pi^{-1}(x))\,d\nu(x).
\]
By property \ref{c:A} of the entropy-preserving topological factor, the integral on the right hand side is equal to zero. This implies $h_\mu(\phi^1)=h_\nu(\psi^1)$.
By the fact that $\phi^1$ extends $\psi^1$ and by \cite[Theorem 17]{Bow:71}, it holds
\begin{equation}\label{extenentropy}
	h_{\rm top}(\psi^1)
	\le h_{\rm top}(\phi^1)
	\le h_{\rm top}(\psi^1) +\sup_{x\in X}h_{\rm top}(\phi^1,\pi^{-1}(x))
	= h_{\rm top}(\psi^1),
\end{equation}
where we again used \ref{c:A}. Hence, both inequalities are in fact equalities. Since $\nu$ is of maximal entropy, it follows
\[
	h_{\rm top}(\Phi)
	= h_{\rm top}(\phi^1)
	= h_{\rm top}(\psi^1)
	= h_{\rm top}(\Psi)
	= h_\nu(\psi^1)
	= h_\mu(\phi^1).
\qedhere\]
\end{proof}

\begin{claim}\label{cla:parrot}
	For every $\mu'\in \cN$ such that $\pi_\ast\mu'=\nu$, it holds $\mu'(Y_0)=1$, where 
\begin{equation*}
Y_0\eqdef \{y\in Y\colon \pi^{-1}(\pi(y))=\{y\}\}
\end{equation*}
 is a Borel set such that $\pi(Y_0)$ is also Borel. 
\end{claim}

\begin{proof}
	By definition, $\pi^{-1}(\pi(y))=\{y\}$ for all $y\in Y_0$. In particular, $Y_0$ is a Borel set and $\pi^{-1}(\pi(Y_0))= Y_0$. 
Furthermore,
\begin{equation*}
  \pi(Y_0) =  \{x\in X: \pi^{-1}(x) \text{ is a singleton}\}=\bigcap_{n=1}^\infty \{x\in X: \text{diam}(\pi^{-1}(x))\le 1/n \}
\end{equation*}
is clearly Borel. 
From hypothesis \ref{c:B} of the entropy-preserving topological factor and $\nu=\pi_\ast\mu'$, it follows
\[
	1
	= \nu(\pi(Y_0))
	= \pi_\ast\mu'(\pi(Y_0))
	= \mu'(\pi^{-1}(\pi(Y_0)))
	= \mu'(Y_0),
\]	
proving the first property.
\end{proof}

\begin{claim}\label{cla:parrotd}
	If $\mu'\in \cN$ and $\mu'\neq\mu$, 
then $h_{\mu'}(\phi^1)<h_\mu(\phi^1)$.
\end{claim}

\begin{proof}
Note that $\nu'\eqdef\pi_\ast\mu'$ is ergodic. If $\nu'=\nu$, then since $\pi^{-1}$ is injective over a $\nu$-full measure set by \ref{c:B}, it would follow $\mu'=\mu$. Hence $\nu'\neq\nu$. By hypothesis, $\nu$ is the unique measure of maximal entropy in $\pi_\ast\cN$, hence it  follows $h_{\nu'}(\psi^1)<h_\nu(\psi^1)$.
By \cite{LedWal:77}, together with hypothesis \ref{c:A}, it holds
\[
	h_{\mu'}(\phi^1)
	\le h_{\nu'}(\psi^1)+\int_Xh_{\rm top}(\phi^1,\pi^{-1}(x))\,d\nu'(x)
	= h_{\nu'}(\psi^1).
\]
The claim now follows from Claim \ref{cla:parrotc}.
\end{proof}	

Note that Claim \ref{cla:parrotd} yields Proposition \ref{prothe:2}.
\end{proof}

\begin{proposition}\label{procorprothe:2}
	Assume that $\Psi$ is a time- and entropy-preserving topological factor of $\Phi$ on $\cN\subset\cM_{\rm erg}(\Phi)$ through a continuous surjective map $\pi\colon Y\to X$.  If $\mu$ is the unique measure of maximal entropy (with respect to $\Phi$) and $\mu\in\cN$, then $\pi_\ast\mu$ is the unique measure of maximal entropy (with respect to $\Psi$).
\end{proposition}

\begin{proof}
As $\mu\in\cN$ is of maximal entropy for $\Phi$, arguing as in Claim \ref{cla:parrotc}, we get
\[
	h_{\rm top}(\Phi)
	= h_\mu(\phi^1)
	= h_{\pi_\ast\mu}(\psi^1)
	= h_{\rm top}(\Psi).
\]
Hence, $\pi_\ast\mu\in\pi_\ast\cN$ is of maximal entropy for $\Psi$. Arguing by contradiction, if there is a measure of maximal entropy $\nu'\ne\pi_\ast\mu$, then there is a measure $\mu'\ne\mu$ of maximal entropy for $\Phi$ satisfying $\nu'=\pi_\ast\mu'$. Contradiction. 
\end{proof}

\section{Time-changes}\label{sec:timcha}

In this section, let $(Y,d)$ be a compact metric space and $\Phi\colon Y\times\bR\to Y$ a continuous flow. Consider a positive continuous function $r\colon Y\to\bR_{>0}$.

Define the function $\ell\colon Y\times\bR\to\bR$ by requiring that for $y\in Y$ and $t\in\bR$ it holds
\begin{equation}\label{eq:elldefori}
	t
	= \int_0^{\ell(y,t)} r(\phi^s(y))\,ds.
\end{equation}
Note that the function $(y,t)\mapsto \ell(y,t)$ is continuous. Moreover, for every $y\in Y$ and $t_1,t_2\in\bR$ it satisfies  $\ell(y,t_1+t_2)=\ell(y,t_1)+ \ell(\phi^{t_1}(y),t_2)$. Hence, for all $y\in Y$ and $t\in\bR$ we have
$\ell(y,-t)=-\ell(\phi^{-t}(y),t)$, $\ell(y,0)=0$, and $\ell(y,t)> 0$ if $t> 0$.  Furthermore, for $y\in Y$ fixed, the function $t\mapsto\ell(y,t)$ is strictly increasing and it holds
\[
	\lim_{t\to\pm\infty}   \ell(y,t)= \pm\infty.
\]
Define  $\Phi_r\colon Y\times\bR\to Y$ by
\begin{equation}\label{eq:elldef}
	\Phi_r(y,\tau)=\phi^{\tau}_r(y)
	\eqdef \phi^{\ell(y,\tau)}(y), \quad \text{for }(y,\tau)\in Y\times\bR.
\end{equation}
We also write $\Phi_r=(\phi_r^\tau)_{\tau\in\bR}$. Observe that $\Phi_r$ is a continuous flow, known as the \emph{time-change of $\Phi$ by $r$}.
Note that for every $y\in Y$ the orbits of $y$ with respect to $\Phi$ and $\Phi_r$, that is, $(\phi^t_r(y))_{t\in\bR}$ and $(\phi^t(y))_{t\in\bR}$ coincide and have the same orientation. Moreover, $\Phi$ turns out to be the time-change of $\Phi_r$ by the function $1/r$, that is, we define the map $(y,t)\mapsto k(y,t)$ so that for every $y\in Y$ and $t\in\bR$ it holds
\[
	t
	= \int_0^{k(y,t)} \frac{1}{r(\phi^s_r(y))}\,ds.
\]
Then for every $y\in Y$ and $t\in\bR$ we have
\begin{align}
	\Phi(y,t)&=\phi^t(y)
	=\phi_r^{k(y,t)}(y), \label{eq:kdef}\\
	t
	&= k(y,\ell(y,t))
	= \ell(y,k(y,t)). \label{eq:krel}
\end{align}

\begin{remark}
Suppose that $Y$ is a manifold and the flow $\Phi$ is generated by the vector field $V$, that is, the curve $t\mapsto y(t)=\phi^t(y_0)$ $(t\in\bR)$, is the unique solution of the equation $dy/dt = V(y)$ with the initial condition $y(0)=y_0$. In this case, one can view the flow $\Phi_r$ as the flow generated by the vector field $W(y)=r(y)V(y)$. Since $W$ is a scalar multiple of $V$ the solutions curves stay the same, but the time parametrization (the speed along the curve) changes.
\end{remark}

\subsection{Bijection between spaces of (ergodic) invariant measures}

The time-changed flow $\Phi_r$ evolves at a different speed and \emph{a priori} possesses different invariant probability measures.
Given $\mu\in\cM_{\rm erg}(\Phi)$, the measure $\mu_r$ is a measure absolutely continuous with respect to $\mu$ with the Radon-Nikodym derivative given by
\begin{equation}\label{sua}
	\frac{d\mu_r}{d\mu} \eqdef \frac{r}{\int r\,d\mu}.
\end{equation}
The measure $\mu_r$  is an ergodic $\Phi_r$-invariant probability measure (\cite[Theorem 5.1]{Tot:66} or \cite[Theorem 3.5.4]{FisHas:19}).
The map $\mu\mapsto\mu_ r$ is a bijection between $\cM_{\rm erg}(\Phi)$ and $\cM_{\rm erg}(\Phi_r)$. By \eqref{sua}, for every continuous function $\varphi\colon Y\to\bR$ we have
\begin{equation}\label{defahh}
  \int_M \frac{\varphi}{ r} \,d\mu_ r
  = \frac{\int_M\varphi \,d\mu}{\int_M r \,d\mu}\,.
\end{equation}
See~\cite[Chapter 10]{CorFomSin:82}, \cite{Tot:66}, or \cite[Chapter 3.5]{FisHas:19} and references therein for more details.

For every $\mu\in\cM_{\rm erg}(\Phi)$, given $t>0$ such that $\mu$ is $\phi^t$-ergodic, by Abramov's formula \eqref{abra} together with  \cite[Theorem 10.1]{Tot:66}) we have
\begin{equation}\label{eq:AbTo}
	h_\mu(\Phi)
	=h_\mu(\phi^1)
	=\frac{1}{t}h_\mu(\phi^t)
	=\frac1th_{\mu_r}(\phi_r^t)\int r\,d\mu
	=h_{\mu_r}(\Phi_r)\int r\,d\mu.
\end{equation}

\subsection{Synchronizing time-changes}\label{secsyntimcha}

In this section, we use the particular time-change $r$ defined by \eqref{eq:T-for-Pnew} that ``synchronizes'' thermodynamical properties to prove Theorem \ref{thepro:2}.

We are now ready to show that for any hyperbolic potential an appropriately chosen time-change synchronizes any equilibrium state into a measure of maximal entropy. 

\begin{proof}[Proof of Theorem \ref{thepro:2}]
The function $r\colon Y\to\bR$ in \eqref{eq:T-for-Pnew} is continuous and positive.  Together with \eqref{eq:P} and standard properties of topological pressure (see \cite[Chapter 9.2]{Wal:82})  we have
\begin{equation}\label{eq:zero}
	P_{\rm top}(\Phi,-r)
	= 	P_{\rm top}\Big(\Phi,\frac1T\barf{T}- P_{\rm top}(\Phi,f)\Big)
	= 	P_{\rm top}(\Phi,\frac1T\barf{T})- P_{\rm top}(\Phi,f)
	= 0.
\end{equation}
Let us consider the time-change of $\Phi$ by $r$. For every $\mu\in\cM(\Phi)$ there is a unique associated measure $\mu_r\in\cM(\Phi_r)$ defined by~\eqref{sua}. Note that $r>0$ implies that the function $\bR\ni q\mapsto P_{\rm top}(\Phi,-q\,r)$ is monotonically decreasing. Hence by \eqref{eq:zero}, $q_0=1$ is its unique zero.

\begin{claim}
A $\Phi$-invariant measure $\bar\mu$ is an equilibrium state  for the potential $f$ (with respect to $\Phi$) if, and only if, $\bar\mu$ is an equilibrium state for the potential $-r$ (with respect to $\Phi$). 
\end{claim}

\begin{proof}
 Since $\frac1T\barf{T}=P_{\rm top}(\Phi,f)-r$, from \eqref{eq:int-of-barphi} we see that $\bar\mu$ is an equilibrium state for $f$ (with respect to $\Phi$) if and only if
\begin{equation}\label{eq:IF}
  P_{\rm top}(\Phi,f)
	= h_{\bar\mu}(\phi^1)+\int f\,d\bar\mu
	= h_{\bar\mu}(\phi^1)+\int(P_{\rm top}(\Phi,f)-r)\,d\bar\mu.
\end{equation}
On the other hand, \eqref{eq:zero} implies that $\bar\mu$	is an equilibrium state for the potential $-r$ (with respect to $\Phi$)
if and only if
\begin{equation}\label{eq:ONLY-IF}
	P_{\rm top}(\Phi,-r)
	=0
	= h_{\bar\mu}(\phi^1)+\int(-r)\,d\bar\mu.
\end{equation}
Since \eqref{eq:IF} and \eqref{eq:ONLY-IF} are equivalent, the claim follows.
\end{proof}

\begin{claim}
	We have $P_{\rm top}(\Phi,-r)=0=P_{\rm top}(\Phi_r,-1)$.
\end{claim}

\begin{proof}
Let $\mu$ be a $\Phi$-invariant ergodic measure and let $\mu_r$ be the associated $\Phi_r$-invariant ergodic measure as defined in~\eqref{sua}.
By \eqref{defahh} and Abramov's formula for entropy~\eqref{eq:AbTo} we have 
\[
	h_{\mu}(\phi^1)-\int r\,d\mu
	= \left(h_{\mu_r}(\phi_r^1)
		- \int 1\,d\mu_r\right)\cdot\left(\int r\,d\mu\right).
\]
Using that $r>0$, the fact that $\mu\mapsto\mu_r$ is a bijection between $\Phi$-ergodic and $\Phi_r$-ergodic measures, and by the variational principle \eqref{varprinc}, it follows
\[\begin{split}
	P_{\rm top}(\Phi_r,-1)
	&= \sup_{\mu_r}\left(h_{\mu_r}(\phi_r^1)-\int1\,d\mu_r\right)
	\ge \frac{1}{\max r}\sup_\mu\left(h_\mu(\phi^1)-\int r\,d\mu\right)\\
	&= \frac{1}{\max r} P_{\rm top}(\Phi,-r)
	=0.
\end{split}\]
On the other hand
\[\begin{split}
	0
	&= P_{\rm top}(\Phi,-r)
	=\sup_\mu\left(h_\mu(\phi^1)-\int r\,d\mu\right)\\
	&\ge \min r\cdot \sup_{\mu_r}\left(h_{\mu_r}(\phi_r^1)-\int 1\,d\mu_r\right)
	\ge \min r\cdot P_{\rm top}(\Phi_r,-1).
\qedhere\end{split}\]
\end{proof}

Furthermore, $\bar\mu$ is an ergodic equilibrium state for the potential $-r$ (with respect to $\Phi$) if, and only if,  $\bar\mu_r$ is an ergodic equilibrium state for the constant potential $f\equiv-1$ (with respect to $\Phi_r$). Equilibrium states for the constant potential $f\equiv-1$ coincides with measures of maximal entropy (with respect to $\Phi_r$). Summing up, to any equilibrium state $\bar\mu$ for $f$ (with respect to $\Phi$) the associated time-changed measure $\bar\mu_r$ is a measure of maximal entropy for $\Phi_r$. Since the time-change $r$ induces a bijection between the two spaces of ergodic probability measures, $\cM_{\rm erg}(\Phi)$ and $\cM_{\rm erg}(\Phi_r)$, we conclude that the latter measure is unique if and only if the former one is unique. This proves the theorem.
\end{proof}

\begin{remark}
It is an immediate consequence of \eqref{defahh} and \eqref{eq:AbTo} that for any  continuous potential $f\colon Y\to\bR$, any $T>0$, and any time-change $r\colon Y\to \bR_{>0}$ we have
\[
	P_{\rm top}\Big(\Phi_r,\frac{\frac1T\barf{T}-P_{\rm top}(\Phi,f)}{r}\Big)
	= 0 .
\]	
Hence, if $f$ is hyperbolic and $T>0$ is such that $r=P_{\rm top}(\Phi,f)-\frac1T\barf{T}>0$, then it holds
\[
	0 = P_{\rm top}\Big(\Phi_r,q\frac{\frac1T\barf{T}-P_{\rm top}(\Phi,f)}{r}\Big)
	= P_{\rm top}(\Phi_r,-q)
\]
if and only if $q=1$. 
\end{remark}

\begin{remark}[Synchronization]
The idea of synchronization is not new. To the best of our knowledge, the first reference is Parry~\cite{Par:86} where he synchronized canonical measures of a certain Anosov flow to obtain a new Anosov flow with constant Lyapunov exponents. See also \cite[Section 2]{Sam:14}. The idea here is similar, even though we cannot guarantee that the time-changed flow is smooth enough to talk about Lyapunov exponents.
\end{remark}

\subsection{Time-change in a TMF}
The time-change of a TMF $\Theta_{\Sigma,R}$ by a continuous function $r\colon \Sigma_R\to\bR_{>0}$
	is the TMF $\Theta_{\Sigma,R'}$, where
\begin{equation}\label{defRdash}
	R'(\uv)
	\eqdef \Delta r(\uv)
	= \int_0^{R(\uv)}r(\theta_{\Sigma,R}^s(\uv,0))\,ds
	= \int_0^{R(\uv)}r(\uv,s)\,ds.
\end{equation}
\begin{lemma}\label{lemtimTMF}
With the notation introduced above, if $R,r$ are H\"older continuous and bounded, then $R'$ is H\"older continuous.
\end{lemma}

\begin{proof}
Using  \eqref{eq:elldefori} it is easy to see that \eqref{defRdash} defines a suspension space homeomorphic to the initial one.

Now assume that $R,r$ are both H\"older continuous and bounded.
Without loss of generality, taking the minimal H\"older exponent,
we can assume that there is $\beta\in (0,1)$ so that $R,r$ are simultaneously $\beta$-H\"older continuous. Let
$\|\cdot \|_0$ and ${\rm Hol}_\beta(\cdot)$ denote the $C^0$ norm and the $\beta$-H\"older constant, respectively.
Fix $\uv,\uw\in\Sigma$. Assuming that $R(\uv)\leq R(\uw)$, we have that
\[\begin{split}
&\, |R'(\uv)-R'(\uw)|=\left|\int_0^{R(\uv)}r(\uv,s)\,ds-\int_0^{R(\uw)}r(\uw,s)\,ds\right|\\
&\leq \int_0^{R(\uv)}|r(\uv,s)-r(\uw,s)|\,ds+\int_{R(\uv)}^{R(\uw)}|r(\uw,s)|\,ds\\
&\leq |R(\uv)|\cdot{\rm Hol}_\beta(r) d(\uv,\uw)^\beta+|R(\uv)-R(\uw)|\cdot\|r\|_0\\
&\leq \left[\|R\|_0{\rm Hol}_\beta(r)+\|r\|_0 {\rm Hol}_\beta(R)\right]d(\uv,\uw)^\beta,
\end{split}\]
thus completing the proof.
\end{proof}

\section{Topological dynamics}\label{sec:topdyn}

 Unless stated otherwise, in this section we consider a continuous fixed-point free flow $\Psi\colon  X\times\bR\to X$ on a compact metric space $(X,d)$, that is, $(x,t)\mapsto \Psi(x,t)=\psi^t(x)$ is a continuous map such that $\psi^{t+s}(x)=\psi^s(\psi^t(x))$  for every $x\in X$, for every $s,t\in\bR$, and $\psi^\tau(x)\ne x$ for some $\tau\ne 0$. We denote by $\Psi^{-1}$ the time-reversed flow $\Psi^{-1}(x,t)=\psi^{-t}(x)$. To avoid some trivialities, we assume that $X$ is \emph{flow perfect},
   that is, no open subset of $X$ is homeomorphic to an open interval in $\bR$ \cite[Definition 1.6.8]{FisHas:19}.

\subsection{Expansivity}

The flow $\Psi$ is \emph{expansive} if for every $\varepsilon>0$ there exists $\delta>0$ with the property that if $x \in X$ and $y \in X$ is a point for which there exists an increasing homeomorphism $\rho\colon \bR \to \bR$ satisfying $\rho(0)=0$ and such that
\[
 	d(\psi^{\rho(t)}(y),\psi^t(x)) 
	\leq \delta
	\text{ for every $t \in {\mathbb R}$,}
\]	
then $y=\psi^{t(y)}(x)$ for some $\lvert t(y) \rvert \le \varepsilon$ (see \cite[Theorem 3]{BowWal:72} for equivalent definitions).%
\footnote{Note that the notion we call \emph{expansivity} is also called \emph{Komuro expansivity}, cf. \cite[Section 1.7]{FisHas:19}, and some authors define  \emph{expansivity} in a different way, cf. \cite[Definition 1.7.2]{FisHas:19}. For fixed-point-free flows these definitions are equivalent \cite[Theorem 1.7.5]{FisHas:19}.}

\subsection{Mixing and transitivity}

The flow $\Psi$ is \emph{topologically mixing} if for every nonempty open sets $U,V\subset X$ there exists $T>0$ such that $\psi^t(U)\cap V\ne\emptyset$ for every $t\ge T$.

The flow $\Psi$ is \emph{topologically transitive} if for every nonempty open sets $U,V\subset X$ there is $t\in\bR$ such that $\psi^t(U)\cap V\ne\emptyset$. Recall that, under our hypothesis that $X$ is a flow-invariant perfect compact metric space,  $\Psi$ is topologically transitive if, and only if, there exists a dense orbit (see, for example, \cite[Proposition 1.6.9]{FisHas:19}). 

It follows that in our setting, topological transitivity is invariant under time-changes. Unfortunately, this is no longer true for topological mixing.

\subsection{Local product structure}\label{sec:lps}

Given $x\in X$, define the \emph{strong stable set} of $x$ by
\begin{equation}\label{eq:defstrstaset}
	\cW^\s(\Psi,x)
	\eqdef \{y\in X\colon d(\psi^t(y),\psi^t(x))\to0\text{ as }t\to\infty\}
\end{equation}
and given $\varepsilon>0$ also define
\[
	\cW^\s_\varepsilon(\Psi,x)
	\eqdef \{y\in \cW^\s(\Psi,x)\colon d(\psi^t(y),\psi^t(x))\le\varepsilon\text{ for all }t\ge0\}.
\]
Define the \emph{strong unstable set} by $\cW^\u(\Psi,x)\eqdef \cW^\s(\Psi^{-1},x)$ and let $\cW^\u_\varepsilon(\Psi,x)\eqdef \cW^\s_\varepsilon(\Psi^{-1},x)$.
Observe that for $x\in X$ and $\ddag \in\{\s,\u\}$ we have
\begin{equation}\label{eq:extmanifol}
	\cW^\ddag(\Psi,y)
	=\cW^\ddag(\Psi,x)
	\quad\text{ for every }y\in\cW^\ddag(\Psi,x),
\end{equation}
and
\begin{equation}\label{eq:invmanifol}
	\psi^t\big(\cW^\ddag(\Psi,x)\big)
	=\cW^\ddag(\Psi,\psi^t(x))
	\quad\text{ for every }t\in\mathbb{R}.
\end{equation}

The flow $\Psi$ has a \emph{local product structure} if for every $\varepsilon>0$ there is $\delta>0$ such that for every $x,y\in X$ with $d(x,y)\le\delta$ there is a unique $\tau=\tau(x,y)\in\bR$ with $\lvert\tau\rvert\le\varepsilon$ such that the set
\[
	 \cW^\s_\varepsilon(\Psi,\psi^\tau(x))\cap \cW^\u_\varepsilon(\Psi,y)
\]
is nonempty and consists of a single point that we denote $[x,y]$.

The following result is an immediate consequence of the uniform continuity of $\Psi$, we state it for further reference.

\begin{lemma}\label{ref:reverse}
	If $\Psi$ has a local product structure, then for every $\varepsilon>0$ there is $\delta>0$ such that for every $x,y\in X$ with $d(x,y)\le\delta$ there is a unique $\tau=\tau(x,y)\in\bR$ with $\lvert\tau\rvert\le\varepsilon$ satisfying
\[
	 \cW^\u_\varepsilon(\Psi,\psi^\tau(x))\cap \cW^\s_\varepsilon(\Psi,y)
	\ne\emptyset.
\]
\end{lemma}

\subsection{Orbit closing property}

A flow $\Psi$ on a compact metric space $X$ has the \emph{orbit closing property} if
for every $\eps > 0$, there exist $\delta>0$ and $T>0$ such that for every $x\in X$
and $t > T$ with $d(x,\psi^t(x))\le\delta$, there exist $y\in X$ and $\ell>0$ satisfying
$|\ell-t|\le\eps$, $\psi^\ell(y) = y$, and $d(\psi^s(x), \psi^s(y))\le \eps$ for $0\le s \le \min\{t,\ell\}$.%
\footnote{Note that some authors use the notion of \emph{flows satisfying closing lemma} (see \cite{CouSch:10, CouSch:14}), which is defined in a slightly different way than above. For a compact flow both versions are easily seen to be equivalent, while for non-compact spaces the flow with the orbit closing property must also  satisfy the closing lemma, but the converse is not necessarily true.}

\subsection{Pseudo-orbit-tracing and closing properties}

Given positive numbers $T$ and $\delta$, we call a sequence of pairs $(x_k,\tau_k)_{k=k_0}^{k_1}$ of points $x_k\in X$ and positive numbers (we allow $k_0=-\infty$ or $k_1=\infty$), a \emph{$(\delta,T)$-pseudo orbit} for the flow $\Psi$ if for every $k\in\bZ$ with $k_0\le k <k_1$ we have $\tau_k\ge T$ and $d(\psi^{\tau_k}(x_k),x_{k+1})\le\delta$. We call a pseudo orbit $(x_k,\tau_k)_{k=-\infty}^{\infty}$ \emph{periodic} if both sequences $(x_k)_{k=-\infty}^{\infty}$ and $(\tau_k)_{k=-\infty}^{\infty}$ are periodic.

Given a sequence $(\tau_k)_{k=-\infty}^{\infty}$, we set $s_0=0$ and for every $k\in\bN$ we define $s_k=\tau_0+\ldots+\tau_{k-1}$, and $s_{-k}=\tau_{-k}+\ldots+\tau_{-1}$.
A $(\delta,T)$-pseudo orbit $(x_k,\tau_k)_{k=-\infty}^\infty$ is \emph{$\varepsilon$-traced} by an orbit $(\psi^t(y))_{t\in \bR}$ of  $y\in X$ if there is some increasing continuous injection with continuous inverse $\rho\colon\bR\to\bR$ satisfying $\rho(0)=0$ and the following:
\begin{enumerate}
  \item for $k=0,1,\ldots$ and for every $s\ge0$ satisfying $s_k\le s<s_{k+1}$ we have
\[
	d(\psi^{\rho(s)}(y),\psi^{s-s_k}(x_k))\le\varepsilon,
\]
  \item for $k=1,2,\ldots$ and for every $s\le0$ satisfying $-s_{-k}\le s< -s_{-k+1}$ we have
\[
	d(\psi^{\rho(s)}(y),\psi^{s+s_{-k}}(x_{-k}))\le\varepsilon.
\]
\end{enumerate}
The flow $\Psi$ has the \emph{pseudo-orbit-tracing property} if for every $\varepsilon>0$  and $T>0$ there is $\delta>0$ such that every $(\delta,T)$-pseudo orbit is $\varepsilon$-traced by a true orbit of $\Psi$. Note that the pseudo-orbit-tracing property is preserved by time-changes. It is also clear that if the flow $\Psi$ has the pseudo-orbit-tracing property, then every finite $(\delta,T)$-pseudo-orbit $(x_k,\tau_k)_{k=k_0}^{k_1}$ can be extended to full pseudo-orbit $(x_k,\tau_k)_{k=-\infty}^\infty$ and traced. By \cite[Thm. 4]{Kom:84}, if the  flow has no fixed points, then we may assume that $\rho$ appearing in the definition of $\eps$-tracing is onto or even that it satisfies
\begin{equation}\label{eq:Komuro}
\left\lvert\frac{\rho(s)-\rho(t)}{s-t}-1\right\rvert<\eps\quad\text{ for every }s,t\in\mathbb{R},\, s\neq t.
\end{equation}
It is also not hard to see that if for every $\varepsilon>0$  there is $\delta>0$ such that every $(\delta,1)$-pseudo orbit is $\varepsilon$-traced by a true orbit of $\Psi$, then $\Psi$ has the pseudo-orbit-tracing property.

By choosing $\rho$ that satisfies \eqref{eq:Komuro} in the definition of the pseudo-orbit-tracing property and using expansiveness, one immediately obtains that
if $\Psi\colon X\times\bR\to X$ is an expansive, fixed point free flow on a compact metric space with the pseudo-orbit-tracing property, then $\Psi$ also satisfies the following weak version of the closing property: for every $\eps > 0$, there exist $\delta>0$ and $T>0$ such that for every $x\in X$ and $t > T$ with $d(x,\psi^t(x))<\delta$, there exist $y\in X$ and $\ell>0$ satisfying
$|\ell-t|<\eps t$, $\psi^\ell(y) = y$, and $d(\psi^s(x), \psi^s(y)) < \eps$ for $0 < s < \min\{t,\ell\}$. Obtaining a stronger conclusion about the period of the periodic point $y$ requires more work (cf. \cite[Remark 6.2.5]{FisHas:19}). In particular, we use some auxiliary results from \cite{Tho:91}, which are valid for flows with the local product structure.

We now are ready to give the following proof.

\begin{proof}[Proof of Theorem \ref{thmprolem:pseudoorbittracing}]
First, by {\cite[Theorem 7.1]{Tho:91}} $\Psi$ has the pseudo-orbit-tracing property.

By hypotheses, $\Psi$ has a local product structure and hence we can apply the following result, which combines \cite[Lemma 3.1 and Proposition 3.2]{Tho:91}.

\begin{claim}\label{cla:preper}
For every $\beta>0$, there exists $\gamma=\gamma(\beta)>0$ such that the following is true: For every $x,y\in X$, every increasing homeomorphism $\rho\colon \bR\to\bR$ satisfying $\rho(0)=0$, and every interval $[T_1,T_2]$ containing $0$ it holds that if $d(\psi^{\rho(s)}(y),\psi^s(x))\le\gamma$ for every $s\in[T_1,T_2]$ then $\lvert\rho(s)-s\rvert\le\beta$ for every $s\in[T_1,T_2]$.
\end{claim}

Given $\varepsilon\in(0,1)$, let $\varepsilon_1\in(0,\varepsilon)$ be such that 
\begin{equation}\label{eq45b}
	d(\psi^s(x),x)\le\frac\varepsilon2
	\quad\text{ for every }x\in X\text{ and }\lvert s\rvert\le\varepsilon_1.
\end{equation}	

Let $\gamma=\gamma(\frac12\varepsilon_1)$ be as provided by Claim \ref{cla:preper}. Use the expansivity property to find  $\delta_1=\delta_1(\frac12\varepsilon_1)>0$. Set $\gamma'=\min\{\varepsilon,\gamma,\delta_1\}$. Let $\delta_0$ be provided by the pseudo-orbit-tracing property associated to $T=1$ and $\frac12\gamma'$. Fix $\delta\in(0,\delta_0)$. To prove the closing property, let $x\in X$ be such that there is $t\ge T=1$ satisfying $d(x,\psi^t(x))\le\delta$. Consider the periodic $(\delta,1)$-pseudo orbit $(x_k,\tau_k)_k$, where $x_k=x$ and $\tau_k=t$ for all $k\in\bZ$. Hence, letting $s_k=kt$ for every $k\in\bZ$, by the tracing property this pseudo-orbit is $\frac12\gamma'$-traced by the orbit of some point $y$ by means of some increasing homeomorphism $\rho\colon\bR\to\bR$ satisfying $\rho(0)=0$ and
\begin{equation}\label{eq:coro}
	d(\psi^{\rho(kt+s)}(y),\psi^s(x))
	\le\frac12\gamma'
	<\gamma
	\quad\text{ for every }k\in\bZ\text{ and every }s\in[0,t).
\end{equation}
By the choice of $\gamma$, it holds
\begin{equation}\label{eq45bb}
	\lvert \rho(s)-s\rvert \le\varepsilon_1/2
	\quad\text{ for every }s\in[0,t].
\end{equation}	
In particular, $\lvert\rho(t)-t\rvert\le\varepsilon_1/2$.
Taking $k=1$, it follows that for all $s\in\bR$ it holds
\[\begin{split}
	d(\psi^{\rho(s)}(y),\psi^{\rho(t+s)}(y))
	&\le d(\psi^{\rho(s)}(y),\psi^s(x))+d(\psi^s(x),\psi^{\rho(t+s)}(y))\\
	\text{\small by \eqref{eq:coro}}\quad
	&\le \frac{\gamma'}{2}+\frac{\gamma'}{2}= \gamma'.
\end{split}\]
Letting $z\eqdef\psi^{\rho(t)}(y)$, we see that for all $s\in\bR$ it holds
\[
d(\psi^{\rho(s)}(y),\psi^{\rho(t+s)-\rho(t)}(z))
\le\gamma'\le\delta_1.
\]
Observe that $r(\tau)\eqdef\rho(t+\rho^{-1}(\tau))-\rho(t)$ defines a homeomorphism of $\bR$ satisfying $r(0)=0$ and such that for every $\tau\in\bR$ it holds
\[
d(\psi^\tau(y),\psi^{r(\tau)}(z))
\le\delta_1.
\]
By expansivity, we have $z=\psi^{t(z)}(y)$ for some $\lvert t(z)\rvert\le\frac12\varepsilon_1$. It follows
\[
	\rho(t)-t(z)
	=t+\rho(t)-t-t(z)
	\ge t-\frac{\varepsilon_1}{2}-\frac{\varepsilon_1}{2}
= t-\varepsilon_1>0.
\]
Hence, $y$ is periodic with period $\ell=\rho(t)-t(z)>0$ satisfying
\[
-\varepsilon_1
= t-\frac{\varepsilon_1}{2}-\frac{\varepsilon_1}{2}-t
\le \ell-t
= \rho(t)-t(z)-t
\le t+\frac{\varepsilon_1}{2}+\frac{\varepsilon_1}{2}-t
=\varepsilon_1.
\]
In particular, $\lvert \ell-t\rvert\le\varepsilon$.
Moreover, for every $s\in[0,\min\{\ell,t\}]$ it holds
\[\begin{split}
d(\psi^s(y),\psi^s(x))
&\le d(\psi^s(y),\psi^{\rho(s)}(y))+d(\psi^{\rho(s)}(y),\psi^s(x))\\
&=d(\psi^s(y),\psi^{\rho(s)-s}(\psi^s(y)))+d(\psi^{\rho(s)}(y),\psi^s(x))\\
\text{\tiny{(by  \eqref{eq45bb} together with \eqref{eq45b} and \eqref{eq:coro})}}\quad
&\le\frac\varepsilon2+\frac12\gamma'
\le \frac\varepsilon2+\frac\varepsilon2=\varepsilon.
\end{split}\]
This implies the closing property and finishes the proof of the lemma.
\end{proof}

\subsection{Invariants for orbit equivalent flows}\label{secinvariants}

A time-change of a continuous flow $\Psi\colon\bR\times X\to X$ induces an \emph{orbit equivalence}, that is, a homeomorphism of $X$ which preserves orbits and their orientation and only alters their time-parametrization. If a time-change induces a  \emph{topological conjugacy} between the flow and its time-change, then properties such as expansivity and being topologically mixing immediately passes from one to the other. But some properties, like mixing, are not preserved by time-changes in general.

\begin{remark}[Trivial time-changes --- orbit equivalences that are conjugacies]
Assume that $X$ is a compact manifold.	If $R\colon X\to\bR$ is continuous and differentiable with respect to $\Psi$ with derivative $R'=(d/dt)R\circ \Psi$ such that $1+R'>0$ (the derivative in the flow direction is positive), then the flow $\Psi$ and its time-change $\Psi_{1+R'}$, are topologically conjugate by means of the map $\pi\colon X\to X$ defined by $\pi(x)\eqdef \psi^{R(x)}(x)$, that is, $\pi\circ\psi^t=\psi_{1+R'}^t\circ\pi$ for every $t\in\bR$. This kind of time-changes are called \emph{trivial} or \emph{canonical}, see \cite[Example 1.3.20 and Proposition 1.3.21]{FisHas:19}).
\end{remark}

Let us now state some properties that are preserved by general time-changes (that is,  which do not  \emph{a priori} induce a topological conjugacy). Let $r$ be a time-change.

Clearly, $\Psi$ is topologically transitive if and only if its time-change $\Psi_r$ is.

Moreover, expansivity is preserved by a time-change.
Note that Bowen and Walters (\cite[Theorem 3]{BowWal:72}) give several equivalent definitions of expansivity. Since the very definition of expansivity takes into consideration time-changes of orbits, the following fact is an immediate consequence (see also \cite[Corollary 4]{BowWal:72}).

\begin{lemma}\label{lempro:3}
		If the flow $\Psi\colon X\times \bR\to X$ on a compact metric space $(X,d)$  is expansive and $r\colon X\to\bR_{>0}$ is a continuous function, then for every $R>0$ satisfying $\max_Xr\le R$ the time-change $\Psi_r\colon X\times \bR\to X$ is an expansive flow on the metric space $(X,d_R)$, where $d_R(\cdot,\cdot)=Rd(\cdot,\cdot)$.
\end{lemma}

As the pseudo-orbit-tracing property is defined taking into account some possible time-change on the tracing orbit, this property is also preserved under an arbitrary time-change.

\begin{lemma}[{\cite[Proposition 1.4 and Theorem 1]{Tho:82}}]\label{lempro:4}
	The flow $\Psi$ has  the pseudo-orbit-tracing property if, and only if, its time-change $\Psi_r$ does.
\end{lemma}

\begin{remark}\label{rem:noinvariance}
	See \cite[Chapter 10]{CorFomSin:82} for further basic properties of time-changes. The entropy of flows with fixed-points from the point of view of time-changes was studied in \cite{Ohn:80}.

	To preserve properties such as hyperbolicity under a time-change is a more subtle issue. By Anosov--Sinai \cite{AnoSin:67} (see also Parry \cite[Section 6]{Par:86}), the time-change of a $C^1$ Anosov flow $\Psi$  under a $C^1$ function $r$ is again an Anosov flow (and, in particular, the flow $\Psi_r$ has a H\"older continuous hyperbolic splitting). Note that the hyperbolic splitting $E^\s\oplus E^\c\oplus E^\u$ of  an Anosov flow varies H\"older continuously if the flow is $C^2$. In \cite{Par:86}, Parry also considers time-changes by means of the geometric potential $\varphi^{(\u)}$ which is defined as in \eqref{phidef} with $E^\u$ in the place of $F^\u$. In general, an improvement of the regularity of the flow does not improve the regularity of $\varphi^{(\u)}$. This potential is $C^{n-1}$, $n\ge2$, provided that $\Psi$ is $C^n$ and that we \emph{assume} that $v\mapsto E^\u_v$ is $C^{n-1}$. The latter condition is not always guaranteed, even if we start, for example, from an Anosov geodesic flow of a surface of constant negative curvature (see \cite[Remark p. 272]{Par:86}).
	
	The existence of a local product structure is \emph{a priori} not preserved under time-change. Brunella \cite[Section 1]{Bru:95} provides an example of a local flow  with trivial (degenerate to a point) strong stable and strong unstable sets.
	
	The property of being topologically mixing is in general not preserved under time-change. For example, any time-change resulting in a time-$t$ suspension flow is not mixing.
\end{remark}

\section{Geodesic flows of nonpositively curved surfaces}\label{sec:geodesic}

In this section, we discuss geodesic flows and an associated class of hyperbolic potentials. Throughout, consider a $C^\infty$ closed connected surface $M$ of negative Euler characteristic and a Riemannian metric of nonpositive Gaussian curvature. We study ergodic properties of the geodesic flow $G=(g^t)_{t\in\bR}$ on the unit tangent bundle $T^1M$. 

\subsection{Some geometric facts}\label{secrem:basicfacts}

We quickly recall some fundamental properties of un-/stable sets for $G$. 
Geometry naturally provides the existence of two \emph{continuous} line bundles $F^\s$ and $F^\u$ of the tangent space $TT^1M$, defined at every point of $T^1M$, each of them being transverse to the line-bundle which is tangent to the orbit foliation and invariant under the derivative of the geodesic flow. It is know that the only obstruction for the flow being Anosov is when there exists an orbit along which $F^\s$ and $F^\u$ coincide \cite{Ebe:73I-II}.  The potential $\varphi^{(\u)}$ is well-defined and depends continuously on $v$ since the subbundle $F^\u_v$ varies continuously with $v$. These bundles are defined in terms of stable and unstable Jacobi fields (we refrain from giving the details and refer, for example, to \cite{Ebe:01}). Denote by $V$ the vector field that generates the geodesic flow $G$, and let $F^\c$ be the line bundle containing $V$. 

The orbit space of $G$ naturally splits into two invariant subsets as follows. A vector $v\in T^1M$ is \emph{singular} if $F_v^\s$ and $F_v^\u$ are colinear. Otherwise, we say that $v$ is \emph{regular}. Denote by $\cR$ the \emph{set of regular vectors} and by $\cS\eqdef T^1M\setminus\cR$ the \emph{set of singular vectors}. The set $\cS$ is compact and invariant. It can be characterized as the set of vectors tangent to a geodesic that entirely lies in a set of points at which the manifold has zero curvature. As such, $\cS$ is the set without any hyperbolic behavior, while the complementary set $\cR$ exhibits some (possibly nonuniformly) hyperbolic behavior. Denote by $\cR_{\rm per}\subset\cR$ the \emph{subset of all regular vectors tangent to some closed geodesic}. From now on, we assume that $\cR\neq\emptyset$, in which case the surface $M$ is said to be \emph{rank $1$}. We are also mainly interested in the case that $\cS\ne\emptyset$, in which $G$ is not an Anosov flow.

We also write $F^\cs\eqdef F^\c\oplus F^\s$ and $F^\cu\eqdef F^\c\oplus F^\u$. If $M$ has negative curvature, then $G$ is Anosov and $TT^1M=F^\s\oplus F^\c\oplus F^\u$ is the hyperbolic splitting. If $M$ has nonpositive curvature, then the subbundles $F^\s,F^\u$ may intersect. We list some essential properties:
\begin{enumerate}[label=$\circ$]
\item $\dim(F^\s)=\dim(F^\u)=1$,
\item ($G$-invariance) $dg^t_v(F^\ddag_v)=F^\ddag_{g^t(v)}$ for all $v\in T^1M$, $t\in\bR$, and $\ddag\in\{\s,\cs,\u,\cu\}$,
\item $F^\s_v,F^\cs_v,F^\u_v,F^\cu_v$ depend continuously on $v$,
\item $F^\s_v$ and $F^\u_v$ are both orthogonal to $V(v)$,
\item $F^\s_v$ and $F^\u_v$ are colinear if and only if $v\in\cS$.
\end{enumerate}
Furthermore, $F^\ddag$ is integrable to a foliation $\cF^\ddag$ for each $\ddag\in\{\s,\cs,\u,\cu\}$; each leaf is a $C^1$ submanifold. 
We refer to $\cF^\ddag(v)$ for $\ddag\in\{\s,\u\}$ as the \emph{un-/stable leaf} of $v$, respectively. 
The family of sets $\{\cF^\ddag(v)\colon v\in T^1M\}$ for $\ddag\in\{\s,\u\}$ is a minimal foliation, that is, each leaf is dense \cite[Theorem 3.7]{Bal:82}. In fact, the latter property implies that the flow $G$ is topologically mixing (see \cite[Proof of Theorem 3.5]{Bal:82}). By \cite{Bal:82}, regular vectors tangent to closed geodesics are dense in $T^1M$. For every vector $v\in \cR_{\rm per}$, using notation \eqref{eq:defstrstaset}, these leaves are just 
\begin{equation}\label{eqmanifds}
	\cF^\ddag(v)
	= \cW^\ddag(G,v)
	\quad\text{ for }\ddag\in\{\s,\u\}.
\end{equation}
For a proof of \eqref{eqmanifds} see \cite[Proposition 3.10]{BBE-Annals}, see also \cite[Lemma 4.1]{LimPol:}.

\subsection{Some thermodynamic facts}

The geodesic flow $G$ is an example of a dynamical system with interesting thermodynamic properties. The top Lyapunov exponent $\chi^+(\mu)$ provided by the Oseledets ergodic theorem almost everywhere for any ergodic $G$-invariant probability measure $\mu$ coincides with Birkhoff average of this potential, that is, $\chi^+(\mu)=\int\varphi^{(\u)}\,d\mu$.  
The geodesic flow preserves, for example, the Liouville measure. The measure $\widetilde m$ obtained by restricting the Liouville measure to the invariant open set $\cR$ and normalizing to get a probability measure gives a hyperbolic ergodic measure (that is, one which has, besides the Lyapunov exponent $0$ associated to the subbundle tangent to the orbit foliation, one positive and one negative Lyapunov exponent), see  \cite{Pes:77}.

Because the measure $\widetilde m$ is hyperbolic, ergodic, and absolutely continuous with respect to the Liouville measure, the Pesin entropy formula \cite{Pes:77} gives that $\chi^+(\widetilde m)=h_{\widetilde m}(g^1)>0$. On the other hand, any ergodic invariant probability measure $\mu_0$ supported on $\cS$ satisfies 
\begin{equation}\label{eqvpHnhup}
	\chi^+(\mu_0)=h_{\mu_0}(g^1)=0,
	\,\text{ hence }\,
	h_{\rm top}(g^1,\cS)=\sup_{\mu(\cS)=1}h_{\mu}(g^1)
	=0.
\end{equation}	 
The latter equality in \eqref{eqvpHnhup} is just the variational principle for the topological entropy of $g^1|_\cS$ and the supremum is taken over all invariant measures $\mu$ satisfying $\mu(\cS)=1$.
These two facts give rise to a so-called \emph{phase transition}, that is, a \emph{co-existence} of equilibrium states  for $\varphi^{(\u)}$: for every invariant measure $\mu_0$ such that $\mu_0(\cS)=1$ we have
\begin{equation}\label{eq:phatra}
	   h_{\widetilde m}(g^1) - \chi^+(\widetilde m)
	= h_{\mu_0}(g^1) - \chi^+(\mu_0)
	= 0
	= \max_{\mu}\left(h_\mu(g^1)+\int\varphi^{(\u)}\,d\mu\right),
\end{equation}
where the maximum is taken over all ergodic $G$-invariant probability measures $\mu$ (compare Figure \ref{fig.1}). 
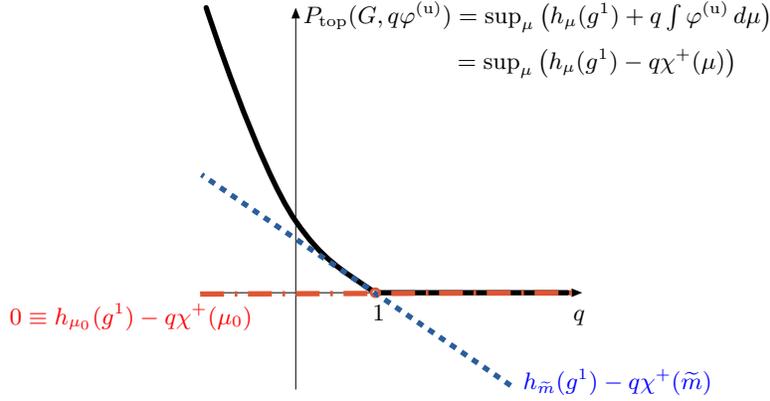
\begin{figure}
\begin{overpic}[scale=.50]{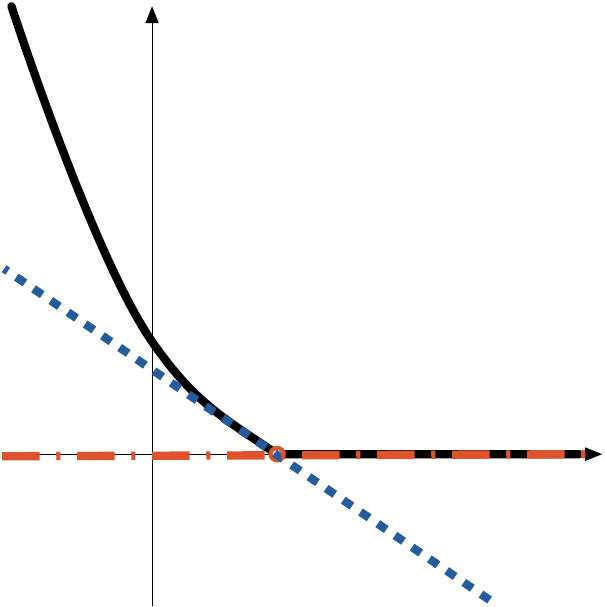}
      	\put(45,18){\small$1$}	
      	\put(97,18){\small$q$}	
		\put(84,0){\small$\textcolor{blue}{h_{\widetilde m}(g^1)-q\chi^+(\widetilde m)}$}
		\put(-49,17){\small$\textcolor{red}{0\equiv h_{\mu_0}(g^1)-q\chi^+(\mu_0)}$}
      	\put(27,95){\small$P_{\rm top}(G,q\varphi^{(\u)})
			=\sup_\mu\big(h_\mu(g^1)+q\int\varphi^{(\u)}\,d\mu\big)$}	
	\put(67,83){\small$=\sup_\mu\big(h_\mu(g^1)-q\chi^+(\mu)\big)$}		
\end{overpic}
\caption{Co-existence of the equilibrium states $\widetilde m$ and $\mu_0$ for the potential $q\varphi^{(\u)}$ at $q=1$}
\label{fig.1}
\end{figure}
Moreover, it is shown in \cite{BurBuzFisSaw:21} that $\widetilde m$ is the only equilibrium state for $\varphi^{(\u)}$ which is not carried on $\cS$.

If $\cS=\emptyset$, then $G$ is an Anosov flow and $F^\u$ is H\"older continuous. Hence, $\varphi^{(\u)}$ is a H\"older continuous function and the classical thermodynamic formalism of hyperbolic dynamical systems guarantees that for any $q\in\bR$ there exists a unique equilibrium state for $q\varphi^{(\u)}$ with respect to the flow \cite[Theorem 3.3]{BowRue:75}.

If $\cS\neq\emptyset$, then equations \eqref{eq:phatra} means that there exist two equilibrium states for $\varphi^{(\u)}$, that is, for the parameter $q=1$. Since the flow $G$ is $C^\infty$, it follows from \cite{Newhouse-Entropy} that the entropy function $\mu\mapsto h_\mu(g^1)$ is upper semi-continuous; see also \cite[Proposition 3.3]{Kni:98}. Hence, for any $q\in\bR$ there exists an equilibrium state for the continuous function $q\varphi^{(\u)}$. By the Ruelle inequality \cite{Rue:78}, for any ergodic $G$-invariant probability $\mu$ we have $h_\mu(g^1)\le\chi^+(\mu)=-\int\varphi^{(\u)}\,d\mu$, hence for every $q\geq 1$ it holds
\[
	\max_{\mu}\left(h_\mu(g^1)+\int q\varphi^{(\u)}\,d\mu\right)
	\le 0.
\]
This implies that every ergodic invariant probability measure $\mu_0$ supported on $\cS$ is an equilibrium state for the potential $q\varphi^{(\u)}$, $q\geq 1$, as
\[
	h_{\mu_0}(g^1)-q\chi^+(\mu_0)
	= 0 - q\cdot 0
	=0
	= \max_\mu\left(h_\mu(g^1)+\int q\varphi^{(\u)}\,d\mu\right).
\]
If $\cS$ contains only one periodic orbit, then $\mu_0$ is the unique equilibrium state for $q\varphi^{(\u)}$,
for any $q\geq 1$. If $\cS$ contains a periodic flat cylinder with a continuum of closed zero curvature geodesics, then the potential $q\varphi^{(\u)}$, $q\geq 1$, has a continuum of coexisting equilibrium states.

\begin{remark}
In \cite{BurChe:24} the authors discuss the weak$\ast$ limit of the equilibrium states as $q\to-1$. 
\end{remark}

\subsection{Hyperbolic potentials}\label{ssechyppot}

\emph{A priori}, hyperbolicity of a potential (recall \eqref{eq:T-for-Pdefhyperbolic}) seems to be a very restrictive hypothesis. However, following~\cite[Proposition 3.1]{InoRiv:12} verbatim, in the case of the geodesic flow $G$, we get the following equivalences which indicate that it is a rather natural assumption. Note that a key property is that, by 
the Ruelle inequality, positive entropy of an ergodic measure (relative to $G$) implies that this measure is hyperbolic.

\begin{lemma}\label{lem:equi}
Let $f\colon T^1M\to\bR$ be continuous. The following facts are equivalent:
\begin{enumerate}[label=\textnormal{(\roman*)}]
		\item The function $f$ is hyperbolic (relative to $G$). \label{c:i}
		\item $P_{\rm top}(G,f)>\max_{\mu\in\cM(G)}\int f \,d\mu$. \label{c:ii}
		\item The metric entropy of each equilibrium state for $f$  (relative to $G$) is  positive. \label{c:iii}
		\item Every continuous function cohomologous%
\footnote{Two continuous functions $r,s\colon T^1M\to\bR$ are \emph{cohomologous} (with respect to $G$) if (using the notation from~\eqref{eq:T-for-Pintegral}) there exists a continuous function $\eta\colon T^1M\to\bR$  such that $\bar{r}_0^t-\bar{s}_0^t=\eta\circ g^t-\eta$ for every $t\in\bR$.
Note that this is equivalent to the fact that $r-s=\lim_{t\to0}(\eta\circ g^t-\eta)/t$.} 
to $f$ is hyperbolic  (relative to $G$). \label{c:iv}
\end{enumerate}
\end{lemma}	

\begin{proof}
Assuming \ref{c:i}, by definition of being a hyperbolic potential, there is $T>0$ satisfying
\[
	\max_{T^1M} \frac1T\barf{T}
	< P_{\rm top}(G,f).
\]
Hence, it follows from the second equality in Lemma \ref{lemequalii} that
\[
	\max_{\mu\in\cM(G)} \int f\,d\mu
	= \inf_{T>0}\max_{T^1M}\frac1T \barf{T}
	< P_{\rm top}(G,f),
\]
and hence \ref{c:ii} holds. Analogously, \ref{c:ii} implies \ref{c:i}.
Assume now \ref{c:ii} and let $\mu$ be an equilibrium state for $f$. Then
\[
	h_\mu(G)+\int f\,d\mu
	= P_{\rm top}(G,f)
	> \max_{\mu\in\cM(G)} \int f\,d\mu,
\]
which implies $h_\mu(G)>0$ and hence \ref{c:iii} holds. Analogously, \ref{c:iii} implies \ref{c:ii}.

To prove the remaining equivalences, let us first consider $s$ cohomologous to $f$. Assume that $\mu$ is an equilibrium state for $f$. Hence there is $\eta\colon T^1M\to\bR$ continuous so that $\barf{T}=\bar{s}_0^T+\eta\circ g^T-\eta$. By Birkhoff's ergodic theorem,
\[\begin{split}
	\int f\,d\mu
	&= \int\lim_{T\to\infty}\frac1T\barf{T}(v)\,d\mu(v)\\
	&= \int\lim_{T\to\infty}\left(\frac1T\bar{s}_0^T(v)
		+\frac1T\eta\circ g^T(v)-\frac1T\eta(v)\right)\,d\mu(v)\\
	&= \int\lim_{T\to\infty}\frac1T\bar{s}_0^T(v)\,d\mu(v)	
	= \int s\,d\mu.
\end{split}\]
This implies that $\mu$ is also an equilibrium state for $s$ and hence \ref{c:iii} holds for $s$. This together with the equivalence of \ref{c:iii} and \ref{c:i} proves that \ref{c:iv} is equivalent to \ref{c:i}.
\end{proof}

Recall the definition of  the geometric potential $\varphi^{(\u)}$ in~\eqref{phidef}.

\begin{lemma}[{\cite[Lemma 4]{GelSch:14}}]\label{lem:11}
	For all $q<1$ it holds
\[
	P_{\rm top}(G,q\varphi^{(\u)})-q\max_{v\in T^1M} \int_0^1\varphi^{(\u)}(g^s(v))\,ds>0,
\]	
hence $q\varphi^{(\u)}$ is hyperbolic with $T=1$. 
\end{lemma}

\subsection{Homoclinic relations}

For $v\in T^1M$, let 
\[
	W^\s(G,v)
	\eqdef \{w\in T^1M\colon \limsup_{t\to\infty}\frac1t\log d(g^t(w),g^t(v))<0\}
\]
and let $W^\u(G,v)$ be defined analogously taking $t\to-\infty$. Define also
\[
	W^\cs(G,v)
	\eqdef \bigcup_{t\in\bR}g^t(W^\s(G,v)),\quad
	W^\cu(G,v)
	\eqdef \bigcup_{t\in\bR}g^t(W^\u(G,v)).	
\]

\begin{remark}
\emph{A priori},  for $\ddag\in\{\cs,\cu\}$, the set $W^\ddag(G,v)$ defined above could be degenerate,
that is, reduce for example just to the vector $\{v\}$. It follows from \eqref{eqmanifds} that 
if $v\in \cR_{\rm per}$ then
\[
	\cF^\ddag(v)
	=  W^\ddag(G,v),
	\quad \ddag\in\{\s,\u\}.
\]
In particular, the equality  $\cF^\ddag(v)=W^\ddag(G,v)$ also holds for $\ddag\in\{\cs,\cu\}.$

It is a consequence of Pesin theory that if $\mu$ is a hyperbolic measure, then for $\mu$-almost every $v$
the set $W^\ddag(G,v)$ is an injectively immersed submanifold which is tangent to $F^{\ddag}$ for  
$\ddag\in\{\c,\cs,\u,\cu\}$.
\end{remark}

Two vectors $v,w\in T^1M$ are \emph{homoclinically related} if 
\[
	W^\cs(G,v)\pitchfork W^\cu(G,w)\ne\emptyset
	\quad\text{ and }\quad
	W^\cu(G,v)\pitchfork W^\cs(G,w)\ne\emptyset.
\]
Here $\pitchfork$ stands for transverse intersections, that is, there are $v_1\in W^\cs(G,v)\cap W^\cu(G,w)$ and $v_2\in W^\cu(G,v)\pitchfork W^\cs(G,w)$ such that $T_{v_1}M=T_{v_1}W^\cs(G,v)\oplus T_{v_1}W^\cu(G,w)$ and $T_{v_2}M=T_{v_2}W^\cu(G,v)\oplus T_{v_2}W^\cs(G,w)$.

Let $\cM_{\rm hyp}(G)\subset\cM_{\rm erg}(G)$ denote the set of all hyperbolic ergodic measures. Two hyperbolic ergodic measures $\mu_1,\mu_2$ are \emph{homoclinically related} if $\mu_1$-almost every vector and $\mu_2$-almost every vector are homoclinically related. By \cite[Proposition 10.1]{BuzCroLim:}, this defines an equivalence relation on $\cM_{\rm hyp}(G)$. In the context of geodesic flows on rank 1 surfaces this equivalence relation has a single equivalence class. 

\begin{corollary}\label{corhomrelmeas}
	All hyperbolic ergodic measures are homoclinically related.
\end{corollary}

\begin{proof}
We only sketch the arguments. First observe that two hyperbolic periodic measures are homoclinically related. This is known to experts in the area and follows from classical results discussed in Section \ref{secrem:basicfacts}. For details, see, for example, \cite[Theorem 1.1]{LimPol:}. 

By a flow version of Katok's horseshoe approximation theorem (see \cite[Supplement S.5]{KatHas:95}), every hyperbolic ergodic measure $\mu$ can be weak$\ast$ approximated by measures supported on hyperbolic periodic orbits. Indeed, in this proof, pseudo-orbits of $\mu$-generic points are shadowed by true orbits. In particular, recurrent $\mu$-generic points can be chosen such that they are shadowed by hyperbolic periodic orbits. By the previous paragraph and the transitivity of the homoclinic relation,
all periodic measures are homoclinically related. This implies the assertion.
\end{proof}

\subsection{Time-preserving symbolic extension and its time-changes}\label{ssectime-preserving-extension}

In this section, we use time-preserving and entropy-preserving symbolic extensions to investigate the geodesic flow and its equilibrium states of the scaled geometric potential.

In order to lift homoclinically related measures to a same irreducible TMF, we use \cite[Main Theorem]{LimPol:}, which we restate in Theorem \ref{thm-coding-irreducible} below in the context of this section.

Recall from Section \ref{sect-MME-symbolic} that, given a TMS $(\Sigma,\sigma)$ and $R\colon\Sigma\to\bR_{>0}$ bounded away from zero and infinity, we define the suspension space $\Sigma_R$ and the suspension flow $\Theta=(\theta^t)_t=\Theta_{\Sigma,R}$. The following result combines \cite[Main Theorem]{LimPol:} with Corollary \ref{corhomrelmeas}.

\begin{theorem}\label{thm-coding-irreducible}
Let $\mu_1,\mu_2$ be two hyperbolic ergodic probability measures. There is an irreducible countable TMF $(\Sigma_R,\sigma_R)$, a subset $\Sigma_R^\#\subset\Sigma_R$, and a H\"older continuous map $\pi_R\colon \Sigma_R\to T^1M$ such that:
\begin{enumerate}[label=\textnormal{(\arabic*)}]
\item $R\colon\Sigma\to\bR^+$ is H\"older continuous and bounded away from zero and infinity,
\item $\pi_R\circ\Theta^t=\varphi^t\circ\pi_R$ for all $t\in\bR$,
\item $\pi_R(\Sigma^\#_R)$ has full measure with respect to $\mu_1$ and $\mu_2$, \label{three}
\item Every $x\in T^1M$ has finitely many pre-images in $\Sigma_R^\#$. \label{four}
\end{enumerate} 
\end{theorem}

Note that in \cite{LimPol:}, the subset $\Sigma_R^\#\subset\Sigma_R$ is given by 
\[
	\Sigma_R^\#
	= \big\{(\uv,t)\in\Sigma_R\colon\exists v,w\in V\,\text{so that }
	v_n=v \text{, }v_{-m}=w\text{ for infinitely many } n,m\in\bN\big\},
\]
where $V$ denoted the vertex set of the TMS. By the Poincaré recurrence theorem, every $\Theta$-invariant probability measure is carried by $\Sigma_R^\#$. This, in particular, implies item \ref{three} in the above theorem. Recalling our definition in Section \ref{sect-MME-symbolic}, the following is now an immediate consequence of Theorem \ref{thm-coding-irreducible} together with Corollary \ref{corhomrelmeas}.

\begin{corollary}\label{corrr}
Let $\mu_1,\mu_2$ be two hyperbolic ergodic probability measures. Then there is an irreducible countable TMF $(\Sigma_R,\sigma_R)$ with a roof function $R$ which is H\"older continuous and bounded away from zero and infinity and there is a H\"older continuous map $\pi_R\colon \Sigma_R\to T^1M$ such that $G$ is a finite-to-one factor of $\Theta$ for $\mu_1$ and $\mu_2$ by $\pi_R$.
\end{corollary}

We can now give the

\begin{proof}[Proof of Theorem \ref{theoremgeodesic} using approach \ref{aa1}]
Fix $q<1$ and let $f=q\varphi^{(\u)}$. By Lemma \ref{lem:11}, $f$ is hyperbolic.
Since $G$ is a $C^\infty$ flow, the existence of an equilibrium state follows from \cite{Newhouse-Entropy}.
We prove uniqueness. To reach a contradiction, assume that there are two distinct ergodic equilibrium states $\mu_1,\mu_2$ for $f$. Since $f$ is hyperbolic, $\mu_1,\mu_2$ are hyperbolic. By Corollary \ref{corrr}, there is an irreducible countable TMF $(\Sigma_R,\sigma_R)$ with a roof function $R$ which is H\"older continuous and bounded away from zero and infinity and there is a H\"older continuous map $\pi_R\colon \Sigma_R\to T^1M$ such that $G$ is a finite-to-one factor of $\Theta_{\Sigma,R}$ for $\mu_1$ and $\mu_2$ by $\pi_R$.

Applying Proposition \ref{prop-finite-to-one-extension} to $\mu_1,\mu_2$, there exist equilibrium states $\nu_1,\nu_2$ for $f\circ\pi_R$ that satisfy $(\pi_R)_\ast\nu_i=\mu_i$, $i=1,2$.
Although it is unknown whether $f$ is Hölder continuous, by \cite[Claim in the proof of Theorem 1.2]{LimPol:} its lift $f\circ\pi_R$ is Hölder continuous with respect to the Bowen-Walters metric 
considered in Section \ref{sect-MME-symbolic}.
Applying Theorem \ref{thepro:2} and Lemma \ref{lemtimTMF}, we obtain a time-changed TMF $\Theta_{\Sigma,R'}$ of $\Theta$ such that $R'$ is Hölder continuous, bounded away from zero and infinity and for which $\nu_1,\nu_2$ are both measures of maximal entropy. Since $\Sigma$ is irreducible, this contradicts Corollary \ref{corTMFmme}. The proof is complete.
\end{proof}

\subsection{Time-preserving factor and its time-changes}\label{sect-expansive-factor}

In this section, we investigate the geodesic flow and, in particular, equilibrium states of its scaled geometric potential, by considering time-preserving and entropy-preserving topological factors.

Denote by $\cM_{\rm erg}^+(G)\subset\cM_{\rm erg}(G)$ the set of all ergodic measures with positive entropy. 
As observed in \cite[Section 6]{GelRug:19}, for every $\mu\in\cM_{\rm erg}^+(G)$, for $\mu$-almost every $v$, it holds $\pi^{-1}(\pi(v))=\{v\}$ and hence $h_{\rm top}(g^1,\pi^{-1}(\pi(v)))=0$. The measure of maximal entropy for $G$ has positive entropy and hence it belongs to $\cM_{\rm erg}^+(G)$. We get the following.

\begin{proposition}[{\cite[Theorem A and Section 6]{GelRug:19}%
\footnote{In fact, in~\cite{GelRug:19} a more general setting of a surface without focal points is considered.}}] \label{pro:mmmain}
	There is a continuous flow $\Psi\colon X\times\bR \to X$ of a $3$-manifold $X$ that is a time- and entropy-preserving topological factor of $G$ on $\cM_{\rm erg}^+(G)$. Moreover, $\Psi$ is topologically mixing, expansive, and has a local product structure.
\end{proposition}

To proceed, we fix a parameter $q<1$ and consider the time-change $G_{ r}$ of $G$ by the continuous function $ r\colon T^1M\to (0,\infty)$ given by%
\footnote{This time-change depends on the parameter $q$, but we will not indicate this dependence in our notation.}  
\begin{equation}\label{fix:R}
	 r
	\eqdef P_{\rm top}(G,q\varphi^{(\u)})-q\max_{v\in T^1M} \int_0^1\varphi^{(\u)}(g^s(v))\,ds,
\end{equation}
Since $\pi$ preserves the time-parametrization, the time-changed flow $G_{ r}$ is an extension of a time-change $\Psi_{\widetilde r}$ of the factor flow $\Psi$ of $G$, where
\begin{equation}\label{deftilder}
	\widetilde r
	=  r\circ\pi^{-1}\colon X\to\bR_{>0},
	\quad
	\text{ with $ r$ as in \eqref{fix:R}}
\end{equation}
Note that $\varphi^{(\u)}(v)=0$ for every $v\in\cS$ and that the set $\pi^{-1}(\pi(v))$ may contain more than one point only if $v\in\cS$. Hence $\widetilde r$ is compatible with the time-preserving topological factor provided by Proposition \ref{pro:mmmain}, that is, the diagram given in \eqref{eqdiagram} commutes. 

\begin{corollary}\label{coriff}
	The potential $q\varphi^{(\u)}$ has a unique equilibrium state (with respect to $G$) if, and only if, $G_{ r}$ has a unique measure of maximal entropy.
\end{corollary} 

\begin{proof}
This is a consequence of Theorem \ref{thepro:2} together with Lemma \ref{lem:11}.
\end{proof}

Note that the regular set $\cR$ satisfies $\cR\subset\cR_0\eqdef\{v\in T^1M\colon\pi^{-1}(\pi(v))=\{v\}\}$. Recall also that the singular set $\cS=T^1M\setminus\cR$ is a closed set that contains only vectors tangent to a geodesic lying  entirely in a set of zero curvature. Note that $\cR_0$ and its complementary set $\cS\setminus\cR_0$ both are $G$-invariant and hence $G_{ r}$-invariant. The projection of these sets by $\pi$ to $X$ are  $\Psi$-invariant and thus $\Psi_{\widetilde r}$-invariant. 

\begin{proposition}\label{proGrtopolofactor}
	Let
\[
	\cN_{ r}
	\eqdef\{\mu\in\cM_{\rm erg}(G_{ r})\colon \mu(\cS)=0\}.
\]	
The flow $\Psi_{\widetilde r}$ is an entropy-preserving topological factor of $G_{ r}$ on $\cN_{ r}$. That is, for every $\mu\in\cN_r$ the entropy of the push-forward of $\mu$ with respect to $G_{r}$ is the same as entropy of $\mu$ with respect to $G_r$.
\end{proposition}

\begin{proof}
We will prove a slightly stronger statement: that for every $\mu \in \cN_{ r}$,
\begin{enumerate}[label=\normalfont{(\roman*)}]
	\item \label{cd:i} $h_{\rm top}(g_{ r}^1,\pi^{-1}(\pi(v)))=0$ for every (and not just $\mu$-almost every) $v\in T^1M$;
	\item \label{cd:ii} $\pi_\ast\mu(\{\pi(v)\colon \pi^{-1}(\pi(v))=\{v\}\})=1$.
\end{enumerate}
If $\pi^{-1}(\pi(v))=\{v\}$ then \ref{cd:i} easily follows. Otherwise, if $\pi^{-1}(\pi(v))$ contains a point $w\ne v$, then given any lift $\tilde v$ of $v$ and $\tilde w$ of $w$ to the universal cover $\tilde M$, the geodesics $\gamma_{\tilde v}$ and $\gamma_{\tilde w}$ are bi-asymptotic and hence bound a flat strip. In particular, $\varphi^{(\u)}\equiv0$ and the sectional curvature vanishes along the orbits of $v$ and $w$. In other words, the time-change on the set $\pi^{-1}(\pi(v))$ is just a  multiplication of the speed of the vector field by the constant
\[
	r_0
	\eqdef P_{\rm top}(G,q\varphi^{(\u)}).
\]
By \cite[Lemma 6.9]{GelRug:19}, it holds $h_{\rm top}(g^1,\pi^{-1}(\pi(v)))=0$. Thus, together with the previous argument $h_{\rm top}(g_{ r}^1,\pi^{-1}(\pi(v)))=0$, proving claim \ref{cd:i} in this other case.

Let $\nu=\pi_\ast\mu$. By contradiction, suppose that $\nu(\pi(\cR_0))<1$ and hence, by $\Psi_{\widetilde r}$-ergodicity $\nu(\pi(\cS\setminus\cR_0))=1$. Hence, using \eqref{sua} and considering the $\Psi$-ergodic measure $\nu_{1/\widetilde r}$ associated to $\nu$, it follows $\nu_{1/\widetilde r}(\pi(\cS\setminus\cR_0))=1$. As in the proof of \cite[Lemma 6.11]{GelRug:19}, it follows $h_{\rm top}(\psi^1,\pi(\cS))>0$, which leads to a contradiction with \eqref{eqvpHnhup}. This verifies \ref{cd:ii}.
\end{proof}

The following is an immediate consequence of Proposition \ref{pro:mmmain} and Theorem \ref{thmprolem:pseudoorbittracing}.

\begin{corollary}\label{cor:firsts}
	The flow $\Psi$ has the pseudo-orbit-tracing property and the closing property.
\end{corollary}

\begin{corollary} 
	For every $q<1$, the time-changed flow  $\Psi_{\widetilde r}$,  for $\widetilde r$ satisfying \eqref{deftilder}, is a continuous fixed-point free flow which is topologically transitive, expansive, and has the pseudo-orbit-tracing property.  
\end{corollary}	

\begin{proof}
	Time-change preserves transitivity and expansivity. Hence, the factor flow $\Psi_{\widetilde r}$ is also transitive and expansive. Corollary \ref{cor:firsts} together with Lemma \ref{lempro:4} implies that $\Psi_{\widetilde r}$ has the pseudo-orbit-tracing property.
\end{proof}

\begin{proof}[Proof of Corollary \ref{corexpansive}]
	Let $q<1$. By Theorem \ref{theoremgeodesic}, $q\varphi^{(\u)}$ has a unique equilibrium state (with respect to $G$). By Corollary \ref{coriff}, there is a unique measure of entropy with respect to $G_r$. This measure belongs to the set $\cN_r\subset\cM_{\rm erg}(\Psi_{\widetilde r})$ defined in Proposition \ref{proGrtopolofactor}.  By Proposition \ref{proGrtopolofactor}, the factor $\pi\colon T^1M\to X$ is time- and entropy-preserving on $\cN_r$. 
 We now apply Proposition \ref{procorprothe:2} to the flow $G_r$, the factor map $\pi$, and the set $\cN_r$.	
\end{proof}

\bibliographystyle{alpha}

\end{document}